\documentclass[11pt,a4paper]{amsart}
\usepackage{color}
\usepackage[colorlinks=true,urlcolor=blue,linkcolor=blue,citecolor=black]{hyperref}
\usepackage{amscd}
\usepackage{amsmath}
\usepackage{mathtools}
\usepackage{amssymb}
\usepackage{amsthm}
\usepackage{setspace}
\usepackage{enumitem}
\usepackage{latexsym}
\usepackage{stmaryrd}
\usepackage{shuffle}
\usepackage{tabularx}
\usepackage{bm}
\usepackage[demo]{graphicx}
\usepackage{caption}
\usepackage{subcaption}
\usepackage{leftindex}

\DeclareFontFamily{U}{bskma}{\skewchar\font130 }
\DeclareFontShape{U}{bskma}{m}{n}{<->bskma10}{}
\DeclareSymbolFont{bskadd} {U} {bskma}{m}{n}
\DeclareMathSymbol{\boxright} {\mathbin}{bskadd} {"A0}

\usepackage{float}
\usepackage{tikz,tikz-cd}
\usetikzlibrary{snakes, matrix,shapes,arrows,positioning,chains}
\numberwithin{equation}{section}

 \usepackage{hyperref}
\hypersetup{
    colorlinks=true,
    linkcolor=teal,
    filecolor=magenta,      
    citecolor=violet,
    urlcolor=cyan,
    pdfpagemode=FullScreen,
    }
    
    \usepackage[margin=1in]{geometry} 
         
\usepackage[nameinlink,capitalize,noabbrev]{cleveref}
\usepackage[T1]{fontenc}   
\usepackage{comment}
\usepackage[english]{babel}
 \usepackage{xargs}                      
\usepackage{xcolor}  
\usepackage[colorinlistoftodos,textsize=tiny]{todonotes}
\newcommandx{\unsure}[2][1=]{\todo[linecolor=red,backgroundcolor=red!25,bordercolor=red,#1]{#2}}
\newcommandx{\change}[2][1=]{\todo[linecolor=blue,backgroundcolor=blue!25,bordercolor=blue,#1]{#2}}
\newcommandx{\info}[2][1=]{\todo[linecolor=cyan,backgroundcolor=cyan!25,bordercolor=cyan,#1]{#2}}
\newcommandx{\improvement}[2][1=]{\todo[caption={Short note},linecolor=violet,backgroundcolor=violet!25,bordercolor=violet,size=\tiny,#1]{#2}}
\newcommandx{\thiswillnotshow}[2][1=]{\todo[disable,#1]{#2}}
\theoremstyle{definition}
\newtheorem{defi}{Definition}[section]
\newtheorem{rem}[defi]{Remark}
\newtheorem{ex}[defi]{Example}
\newtheorem{nota}[defi]{Notation}
\theoremstyle{plain}
\newtheorem{lemma}[defi]{Lemma}

\newtheorem{theo}[defi]{Theorem}
\crefname{theo}{Theorem}{Theorems}
\newtheorem{prop}[defi]{Proposition}

\newtheorem*{theo*}{Theorem}

\newcommand{\id}{\operatorname{id}}

\newcommand{\A}{\mathcal{A}}
\newcommand{\E}{\mathcal{E}}

\newcommand{\NC}{\operatorname{NC}}
\newcommand{\NCirr}{\mathrm{NC}^{\scriptscriptstyle\mathrm{irr}}}

\newcommand{\uno}{{\bf{1}}}

\addto\captionsenglish{
  \renewcommand{\contentsname}%
    {Table of Contents}%
}

\newenvironment{arb}{\begin{tikzpicture}[baseline,scale=0.5,level distance=7mm,level 1/.style={sibling distance=10mm},level 2/.style={sibling distance=5mm},level 3/.style={sibling distance=3mm},grow=down, font=\scriptsize]
\tikzstyle{ve}=[draw,circle,inner sep=1pt,fill] 
\tikzstyle{vv}=[draw,circle,inner sep=1pt] 
\tikzstyle{vee}=[minimum size=0pt ,inner sep=0pt]}{\end{tikzpicture}}
\newenvironment{arbb}{\begin{tikzpicture}[baseline,scale=0.5,level distance=7mm,level 1/.style={sibling distance=25mm},level 2/.style={sibling distance=5mm},level 3/.style={sibling distance=3mm},grow=down, font=\scriptsize]
\tikzstyle{ve}=[draw,circle,inner sep=1pt,fill] 
\tikzstyle{vv}=[draw,circle,inner sep=1pt] 
\tikzstyle{vee}=[minimum size=0pt ,inner sep=0pt]}{\end{tikzpicture}}

\usepackage{forest}
\forestset{
  decor/.style = {
    label/.expanded = {[inner sep = 0.2ex, font=\unexpanded{\tiny}]right:{$#1$}}
  },
  root/.style = {minimum size = 0.1ex},
  decorated/.style = {
    for tree = {
      circle, fill, inner sep = 0.3ex, minimum size = 1.ex,
      grow' = south, l = 0, l sep = 1.2ex, s sep = 0.7em,
      fit = tight, parent anchor = center, child anchor = center,
      delay = {decor/.option = content, content =}
    }
  },
  default preamble = {decorated, root},
  begin draw/.code={\begin{tikzpicture}[baseline={([yshift=-0.5ex]current bounding box.center)}]},
}

\definecolor{red}{rgb}{1.,0.,0.}
\definecolor{green}{rgb}{0.,1.,0.}
\definecolor{blue}{rgb}{0.,0.,1.}
\definecolor{orange}{rgb}{1.,0.8431372549019608,0.}

\title[Conditionally monotone cumulants via shuffle algebra]{Conditionally monotone cumulants\\ via shuffle algebra}

\author[Adrián Celestino]{Adrián Celestino${}^{\dagger}$}
\address[${}^{\dagger}$]{Institut f\"ur Diskrete Mathematik, Technische Universit\"at Graz, Steyrergasse 30, 8010 Graz, Austria.}
\email{celestino@math.tugraz.at}
\urladdr{https://sites.google.com/view/adriancelestino/}
\author[Kurusch Ebrahimi-Fard]{Kurusch Ebrahimi-Fard${}^{\diamond}$}
\address[${}^{\diamond}$]{Department of Mathematical Sciences, Norwegian University of Science and Technology (NTNU), NO-7491 Trondheim, Norway. Centre for Advanced Study (CAS),  Drammensveien 78, 0271 Oslo, Norway.}
\email{kurusch.ebrahimi-fard@ntnu.no}
\urladdr{https://folk.ntnu.no/kurusche/}

\date{\today}

\begin{document}

\begin{abstract}
In this work we study conditional monotone cumulants and additive convolution in the shuffle-algebraic approach to non-commutative probability. We describe c-monotone cumulants as an infinitesimal character and identify the c-monotone additive convolution as an associative operation in the set of pairs of characters in the dual of a double tensor Hopf algebra. In this algebraic framework, we understand previous results on c-monotone cumulants and prove a combinatorial formula that relates c-free and c-monotone cumulants. We also identify the notion of $t$-Boolean cumulants in the shuffle-algebraic approach and introduce the corresponding notion of $t$-monotone cumulants as a particular case of c-monotone cumulants.
\end{abstract}

\maketitle

\section{Introduction}
\label{sec:intro}



Non-commutative probability extends classical probability theory to situations where random variables do not necessarily commute, offering a framework for comprehending and studying non-commutative distributions. In classical probability, the \emph{independence} of events is a fundamental concept. Conversely, in non-commutative probability, as demonstrated by Muraki in his work \cite{Mur}, a specific set of natural axioms gives rise to five notions of independence: tensor, free, Boolean, monotone, and anti-monotone independence. Each of these notions offers a rich framework for establishing a theory of non-commutative probability. Among the various concepts of independence, \textit{free independence}, first introduced by Voiculescu in the 1980s, has proven to be particularly fruitful. It has forged connections with diverse branches of mathematics, including operator algebras, random matrices, non-commutative stochastic processes, combinatorics, representation theory, and quantum information theory.
\par For each natural notion of independence, a corresponding type of \emph{cumulants} can be defined, serving as a crucial tool in the combinatorial and computational understanding of non-commutative probability. Additionally, extensions of non-commutative probability, motivated by both combinatorial and analytic considerations, have been introduced. Notably, M.~Bo\.zejko and R.~Speicher in \cite{BS91, BLS96} introduced the concept of \emph{conditionally free independence} as a framework that jointly encompasses Boolean and free cumulants.
\par In the work \cite{Franz06}, U.~Franz demonstrated that conditionally free independence also encompasses monotone independence. However, the associativity of the monotone product, as well as the notion of monotone cumulants, cannot be deduced from the conditionally free product and cumulants. To address this shortcoming, T.~Hasebe introduced the concept of \textit{conditionally monotone product} in his work \cite{Has}. This innovation allowed him to define corresponding concepts of \textit{conditionally monotone independence} and \emph{conditionally monotone cumulants}. Hasebe established that these notions comprehensively include the Boolean and monotone independence and cumulants, akin to the conditionally free case. Furthermore, it was demonstrated in \cite{Has} that conditionally monotone convolution incorporates \emph{orthogonal convolution}, introduced by R.~Lenczewski in \cite{Len}.
\par On the contrary, the concept of \emph{Hopf algebra} (refer to \cite{CP21}) proves instrumental in the combinatorial exploration of non-commutative probability. Notably, in a recent series of papers, the second author, in collaboration with F.~Patras (\cite{EFP1,EFP2,EFP3,EFP4}), has advanced a Hopf-algebraic approach to free, Boolean, and monotone cumulants. This framework has proven valuable for a comprehensive study of the relationships between moments and cumulants, as well as the additive convolution product in non-commutative probability.
\par Ebrahimi-Fard and Patras examine a specific Hopf algebra on the double tensor algebra $H:=T(T_+(\A))$ associated with a non-commutative probability space $(\A,\varphi)$, along with a character $\Phi$ on $H$ extending the linear functional $\varphi: \A \to \mathbb{K}$. A crucial observation is that the coproduct on $H$ can be split into a sum of two coproducts, allowing the corresponding non-associative products, denoted $\prec$ and $\succ$, in the dual space $H^*:=\operatorname{Lin}(H,\mathbb{C})$, to define the structure of a non-commutative \textit{shuffle algebra} (\cref{prop:shufflealg}). With the convolution product on $H^*$, the shuffle algebra structure enables the definition of three distinct exponential-type maps and their corresponding logarithms. The latter mappings determine infinitesimal characters $\kappa$, $\beta$, and $\rho$, precisely associated with the families of free, Boolean, and monotone cumulants, respectively (\Cref{thm:mainEFP,thm:linkNCP}). Thus, in this context, the moments-cumulant formulas can be understood in a group-theoretic manner, specifically through the connection between the Lie algebra of infinitesimal characters and the group of characters on $H$. 
\par The combinatorial relations between different cumulants given in \cite{AHLV,CEFPP} can be understood in the shuffle picture through the \textit{shuffle adjoint action}, $\theta_\Psi(\alpha) = \Psi^{*-1}\succ\alpha \prec \Psi$, defined for any infinitesimal character $\alpha$ and any character $\Psi$. More precisely, if $\kappa$ and $\beta$ are the linear forms associated with the free and Boolean cumulants of $\Phi$, then one can deduce from the moment-cumulant relations the identities 
\[
	\beta = \theta_\Phi(\kappa) = \Phi^{*-1}\succ \kappa \prec \Phi
	\quad\text{and}\quad 
	\kappa = \theta_{\Phi^{*-1} }(\beta) = \Phi \succ\beta\prec \Phi^{*-1}.
\]
The combinatorial relations between free and Boolean cumulants in terms of irreducible non-crossing partitions in \cite{AHLV} can be obtained from the above equations by evaluating them on a word $w=a_1\cdots a_n\in \A^{\otimes n}$. Thus, the transition between free and Boolean cumulants can be understood as a linear transformation in terms of a particular action of the group of characters on the Lie algebra of infinitesimal characters.
\par In their work \cite{EFP4}, the authors demonstrated that conditionally free cumulants can be naturally understood in terms of the shuffle adjoint action. To elaborate, considering a conditionally non-commutative probability space $(\A,\varphi,\psi)$ and the associated Hopf algebra $H$ with characters $\Phi$ and $\Psi$ extending the linear functionals $\varphi$ and $\psi, $ respectively, the linear form $\mathsf{K}= \theta_{\Psi^{*-1}}(\beta)$ is linked to the conditionally free cumulants of $(\A,\varphi,\psi)$. Here, $\beta$ is the linear form associated with the Boolean cumulants of $\varphi$. Conditionally free additive convolution is defined in the shuffle picture by transporting the sum of the corresponding conditionally free cumulants in terms of shuffle adjoint action.


\subsection*{Contributions of the paper} 

Motivated by the insights from \cite{EFP4}, we present a novel perspective on conditionally monotone cumulants and conditionally monotone additive convolution, using the shuffle algebra approach to non-commutative probability. More precisely, in \Cref{def:cmoninfchar,prop:46P} we consider the conditionally monotone cumulants for $(\A,\varphi,\psi)$ as infinitesimal character $\mathsf{P}$ on $H$ defined through an extension of the inverse of the pre-Lie Magnus expansion, $\beta = W(\rho)$. This extension encapsulates the relationship between Boolean and monotone cumulants in the shuffle picture.

Building upon the strategy of substituting Boolean cumulants of $\Psi$ with those of $\Phi$ in the definition of monotone convolution, we introduce, in \cref{def:cmonconvC}, an operation on the set of character pairs on $H$. Subsequently, in \cref{prop:cmonconvC}, we establish that this operation precisely characterizes conditionally monotone additive convolution. Exploiting the algebraic relations inherent in the shuffle algebra structure, we not only provide an algebraic viewpoint on several known results from \cite{Has} but also give a new formula in \cref{cor:relation} that establishes a connection between conditionally free and monotone cumulants. 

An additional contribution of this manuscript is the elucidation of the notion of \textit{$t$-Boolean cumulants}, introduced by Bo\.zejko and Wysoczanski in \cite{BW01}, as a 1-parameter interpolation between free and Boolean cumulants. We show that this is a specific case of conditionally free cumulants in the shuffle picture. This insight motivates the introduction of the corresponding notion of \textit{$t$-monotone cumulants}, culminating in \cref{thm:tmonotone}, which provides analogous combinatorial formulas to the $t$-Boolean case.


\subsection*{Organization of the paper} 

Apart from the current introduction, the paper is organized as follows. In \cref{sec:cmono}, we outline basic definitions in non-commutative probability, encompassing various notions of independence and cumulants. Specifically, we state the moment-cumulant formulas defining free, Boolean, and monotone cumulants, along with conditionally free and monotone cumulants and convolutions.

Moving on to \cref{sec:shuffle}, we delve into the double tensor Hopf algebra and the shuffle algebraic structure of its dual, serving as a framework for cumulants in non-commutative probability. We explain how the perspective from this framework comprehends conditionally free cumulants and additive convolution.

In \cref{sec:cmoninfchar}, we present the primary contribution of this work by demonstrating how to define an infinitesimal character that extends conditionally monotone cumulants within the shuffle algebraic framework. Similarly, in \cref{sec:condmonoshuffle}, we detail how to define conditionally monotone additive convolution within this framework. Additionally, we establish certain properties from Hasebe's work \cite{Has} using purely shuffle-algebraic arguments.

Moving to \cref{sec:relation}, we unveil a combinatorial relation expressing conditionally free cumulants in terms of conditionally monotone cumulants. Finally, in \cref{sec:tboolean}, we demonstrate that $t$-Boolean cumulants represent a specific case of conditionally free cumulants within the shuffle framework. We extend this construction to the conditionally monotone case, introducing the notion of $t$-monotone cumulants and proving their respective moment-cumulant relation.


\subsection*{Acknowledgements}   
Adrián Celestino is supported by the Austrian Science Fund (FWF) grant I 6232-N (WEAVE). Kurusch Ebrahimi-Fard is supported by the Research Council of Norway through project 302831 “Computational Dynamics and Stochastics on Manifolds” (CODYSMA). He would also like to thank the Centre for Advanced Study (CAS) in Oslo for its support.


\section{Non-commutative independences and cumulants}
\label{sec:cmono}

Voiculescu's theory of free probability provides the most prominent example of a non-commuta- tive probability theory \cite{NSp}. Its notion of independence, called freeness, is characterized by the complete absence of algebraic relations among non-commutative random variables. Nevertheless, non-commutative probability theory extends beyond freeness, accommodating additional notions of independence such as Boolean and monotone independence.


\subsection{Preliminaries on combinatorics}
\label{ssec:prelim}

We start by recalling the combinatorial objects required to describe the moment-cumulant relation. Denoting $[n] :=\{1,\ldots,n\}$, a \textit{partition of $[n]$} is a collection of non-empty pairwise disjoint subsets $\pi=\{\pi_1,\ldots,\pi_m\}$ of $[n]$ such that $[n] = \pi_1\cup\cdots\cup \pi_m$. The elements $\pi_1,\ldots,\pi_m$ are called the \textit{blocks of $\pi$}. We say that $\pi=\{\pi_1,\ldots,\pi_m\}$ is a \textit{non-crossing partition of $[n]$} if it is a partition of $[n]$ such that for any $1\leq a<b<c<d\leq n$ and $\pi_i$ and $\pi_j$ blocks of $\pi$ such that $a,c\in \pi_i$ and $b,d\in \pi_j$, then we have that $\pi_i=\pi_j$. The set of non-crossing partitions on $[n]$ is denoted by $\NC(n)$. The number of blocks of a non-crossing partition $\pi = \{\pi_1,\ldots,\pi_m\} \in \mathrm{NC}(n)$ is denoted by $|\pi|=m$. If both $1$ and $n$ are in the same block, then we call a non-crossing partition \textit{irreducible}. The set of irreducible non-crossing partitions on $[n]$ is denoted by $\mathrm{NC}^{\scriptscriptstyle\mathrm{irr}}(n)$. The sets of non-crossing and irreducible non-crossing partitions on $[n]$ with $k$ blocks are denoted $\mathrm{NC}_k(n)$ respectively $\mathrm{NC}_k^{\scriptscriptstyle\mathrm{irr}}(n)$. An \textit{interval partition of $[n]$} is a non-crossing partition $\pi\in \NC(n)$ such that its blocks are intervals. Interval partitions are non-crossing and form the lattice $\mathrm{I}(n)$. The above definitions can be easily extended for the general case when replacing $[n]$ by a totally ordered set $A$. In particular, we denote by $\NC(A)$ the set of non-crossing partitions on $A$.
\par There is a natural partial order  $\leq$ on the set $\mathrm{NC}(n)$ called \textit{reversed refinement order}. For $\pi, \sigma \in \mathrm{NC}(n)$, we write $\pi \leq \sigma$ if every block of $\sigma$ is a union of blocks of $\pi$. The maximal element $1_n \in \mathrm{NC}(n)$ in this order is the partition of $[n]$ with a single block, $1_n:=\{\{1,\ldots,n\}\}$. The minimal element is $0_n:=\{\{1\},\{2\},\ldots,\{n\}\}$. 
\par The blocks of a non-crossing partition $\pi=\{\pi_1,\dots,\pi_k\}$ are naturally ordered, i.e., $\pi_j<\pi_i$ if and only if there exist elements $a,b\in\pi_j$ such that $a<c<b$ for any $c \in \pi_i$. A block $\pi_i$  is \textit{outer} if it is minimal for this order, that is, if there are no elements $a,c$ in another block $\pi_j$ such that $a<c<b$ for any $c\in \pi_i$. A block $\pi_i$ is called \textit{inner} if it is not an outer block.
\par To any irreducible non-crossing partition ${\pi \in \mathrm{NC}^{\scriptscriptstyle\mathrm{irr}}(n)}$ one can naturally associate a rooted tree $t(\pi)$ encoding the hierarchy of the nested blocks of $\pi$, called the \textit{nesting tree of $\pi$}. Notice that the root is associated with the unique outer block containing $1,n \in [n]$. If the non-crossing partition is not irreducible, then one associates the \textit{nesting forest of $\pi$} as the forest of rooted trees defined in the following way: consider the block $\pi_1$ of $\pi$ that contains $1=:i_{1,\min}$ and write $i_{1,\max}$ for its maximal element. The set of blocks $\pi^{(1)}:=\{\pi_i\in\pi|\pi_1 \leq \pi_i\}$ forms an irreducible non-crossing partition of the set $\{i_{1,\min},i_{1,\min}+1,\dots,i_{1,\max}\}$. Inductively, having defined the irreducible non-crossing partitions $\pi^{(1)},\ldots,\pi^{(j-1)}$, we consider the block $\pi_j\in\pi$ that contains $i_{j-1,\max}+1=: i_{j,\min}$ and write $i_{j,\max}$ for its maximal element. Then the set of blocks $\pi^{(j)}:= \{\pi_i\in\pi| \pi_j\leq \pi_i\}$ forms an irreducible non crossing partition of the set $\{i_{j,\min},i_{j,\min}+1,\ldots,i_{j,\max}\}$. Iterating this construction, we obtain a sequence of irreducible non-crossing partitions $\pi^{(1)},\ldots,\pi^{(\ell)}$ such that $\pi= \pi^{(1)}\cup\cdots\cup \pi^{(\ell)}$. The partitions $\pi^{(1)},\ldots,\pi^{(\ell)}$ are called the \textit{irreducible components of $\pi$}. Finally, the nesting forest of $\pi$, also denoted by $t(\pi)$ is given by the set of trees $\{t(\pi^{(1)}),\ldots,t(\pi^{(\ell)})\}$.
\par Given a non-crossing partition $\pi\in \NC(n)$, we can extend the natural order of its blocks to a total order. A \textit{monotone non-crossing partition of $[n]$} is a pair $(\pi,\lambda)$ such that $\pi\in \NC(n)$ and $\lambda: \pi\to [|\pi|]$ is an increasing bijection. The set of monotone non-crossing partitions on $[n]$ is denoted by $\mathrm{M}(n)$, and its elements will be denoted by $\pi\in \mathrm{M}(n)$, omitting $\lambda$. The set of monotone irreducible non-crossing partitions on $[n]$ is denoted by $\mathrm{M}^{\scriptscriptstyle\mathrm{irr}}(n)$.
\par Recall that, for any rooted tree $t$, the so-called \textit{tree factorial of $t$}, denoted by $t!$, is computed inductively: $t!:=1$ if $t$ is the single-vertex tree, and $t!= |t|t_1!\cdots t_s!$ in the case that $t$ is constructed by attaching rooted trees $t_1,\ldots,t_s$ to a new common root, with $|t|$ being the number of vertices of $t$. The tree factorial of a forest is the product of the tree factorials of the constituting trees. In the case of the nesting forest of a non-crossing partition $\pi$, the number of monotone non-crossing partitions associated to $\pi$ is given by the quantity $m(\pi) := |\pi|!/t(\pi)!$.


\subsection{Non-commutative probability}
\label{ssec:NCPB}

\begin{defi}
\label{def:ncprobaspace}
An \textit{algebraic (or non-commutative) probability space} $(\A,\varphi)$ consists of an associative unital algebra $\A$ and a linear functional $\varphi: \A \to \mathbb{K}$, mapping the algebra unit to the unit of the underlying base field $\mathbb{K}$, i.e., $\varphi(1_\A)=1_\mathbb{K}$. 
\end{defi}

Let $(\A,\varphi)$ be a non-commutative probability space and suppose $\mathbb{K}=\mathbb{C}$. An element $a\in \A$ is called \textit{random variable}. The \textit{distribution} of a tuple of random variables $(a_1,\ldots,a_k)$ is a linear functional on the complex algebra of non-commutative polynomials, $\mu:\mathbb{C}\langle X_1,\ldots,X_k\rangle \to\mathbb{C}$
$$
	\mu(X_{i_1}\cdots X_{i_n}) := \varphi(a_{i_1}\cdots a_{i_n})
$$ 
for any $n\geq1$ and $i_1,\ldots,i_n \in [k]$.

\smallskip

Several analogues of the notion of independence in classical probability have emerged in non-commutative probability. Notably, Muraki demonstrated in \cite{Mur} that only five such notions of independence satisfy certain natural axioms. Here, we will focus on three of these notions, which will be particularly relevant for our discussion.

\begin{defi}[Non-commutative notions of independence]
	Let $(\A,\varphi)$ be a non-commutative probability space and consider subalgebras $\A_1,\ldots,\A_k$ of $\A$.
	\par i) We say that $\A_1,\ldots,\A_k$ are \textit{freely independent} (or simply \textit{free}) if they are unital subalgebras of $\A$ and if
	$$
		\varphi(a_1\cdots a_n) = 0
	$$
	whenever $n\geq2$, $a_1\in \A_{i_1},\ldots,a_n\in \A_{i_n}$,  $i_r\neq i_{r+1}$ for $1\leq r<n$,  and $\varphi(a_i) = 0$ for any $1\leq i\leq n$.
	\par ii) We say that $\A_1,\ldots,\A_k$ are \textit{Boolean independent} if 
	$$
		\varphi(a_1\cdots a_n) = \varphi(a_1)\cdots \varphi(a_n)
	$$
	whenever $n\geq1$,  $a_1\in \A_{i_1},\ldots,a_n\in \A_{i_n}$, and  $i_r\neq i_{r+1}$ for $1\leq r<n$.
	\par iii) We say that $\A_1,\ldots,\A_k$ are \textit{monotone independent} if
	$$
		\varphi(a_1\cdots a_\ell\cdots a_n) = \varphi(a_\ell)\varphi(a_1\cdots a_{\ell-1}a_{\ell+1}\cdots a_n) 
	$$
	whenever $n\geq1$, $a_1\in \A_{i_1},\ldots,a_n\in \A_{i_n}$,  $i_r\neq i_{r+1}$ for $1\leq r<n$,  and $i_{\ell-1}<i_\ell>i_{\ell+1}$, where we omit the condition $i_{\ell-1}<i_\ell$ if $\ell=1$ and $i_\ell>i_{\ell+1}$ if $\ell=n$.
\end{defi}

\begin{rem}
	\label{rmk:freeRV}
Let $(\A,\varphi)$ be a non-commutative probability space and consider two random variables $a,b\in \A$. We say that $a$ and $b$ are \textit{free} if the unital subalgebras $\A_1$ and $\A_2$ generated by $a$ respectively $b$, are free. Analogous definitions can be made for the cases of Boolean and monotone independence, and the subsequent notions of conditionally free and monotone independence.
\end{rem}

Non-commutative cumulants can be defined for each notion of independence, and they have proven to be a crucial combinatorial tool in the study of non-commutative independent random variables.

\begin{defi}[Cumulants]
\label{def:cumulants}
	Let $(\A,\varphi)$ be a non-commutative probability space. 
	\par i) (\cite{Spe94}) \textit{Free cumulants of $(\A,\varphi)$} form a family of multilinear functionals $\{k_n:\A^n\to\mathbb{C}\}_{n\geq1}$ recursively defined by the moment-cumulant relations
\begin{equation}
\label{eq:freeMCrel}	
	\varphi(a_1\cdots a_n)= \sum_{\pi\in\NC(n)} \prod_{\pi_i \in \pi} k_{|\pi_i|}(a_{1},\ldots,a_{n}| \pi_i).
\end{equation}
\par ii) (\cite{SW97}) \textit{Boolean cumulants of $(\A,\varphi)$} form a family of multilinear functionals \linebreak$\{b_n:\A^n\to\mathbb{C}\}_{n\geq1}$ recursively defined by  the moment-cumulant relations
\begin{equation}
\label{eq:booleanMCrel}
	\varphi(a_1\cdots a_n)= \sum_{\pi\in\operatorname{I}(n)}  \prod_{\pi_i \in \pi} b_{|\pi_i|}(a_{1},\ldots,a_{n}| \pi_i).
\end{equation}
\par iii) (\cite{HS11}) \textit{Monotone cumulants of $(\A,\varphi)$} form a family of multilinear functionals \linebreak$\{h_n:\A^n\to\mathbb{C}\}_{n\geq1}$ recursively defined by the moment-cumulant relations
\begin{equation}
\label{eq:monotoneMCrel}
	\varphi(a_1\cdots a_n)= \sum_{\pi\in\NC(n)} \frac{1}{t(\pi)!} \prod_{\pi_i \in \pi} h_{|\pi_i|}(a_{1},\ldots,a_{n}| \pi_i).
\end{equation}
\end{defi}

\begin{rem}
\label{rmk:monotonepartitions}
i) Recall standard notation, which defines for $\pi=\{\pi_1,\ldots,\pi_m\}\in\NC(n)$
$$
	k_\pi(a_1,\ldots,a_n):=\prod_{\pi_i \in \pi} k_{|\pi_i|}(a_{1},\ldots,a_{n}| \pi_i),
$$
where $k_{|\pi_i|}(a_{1},\ldots,a_{n}| \pi_i) := k_{|\pi_i|}(a_{i_1},\ldots,a_{i_{|\pi_i|}})$, for $\pi_i=\{i_1,\ldots,i_{|\pi_i|}\}$. Analogously for Boolean and monotone cumulants.  Later we will also use word notation  
$$
	w_{\pi_i}=a_{i_1}\cdots a_{i_{|\pi_i|}}, \qquad  \pi_i=\{i_1,\ldots,i_{|\pi_i|}\} \in \pi \in \NC(n).
$$

ii) Let $\pi\in\NC(n)$ and consider its so-called nesting forest $t(\pi)$. Recall that $m(\pi) = |\pi|!/t(\pi)!$ is the number of monotone partitions associated to $\pi$. Thus, the moment-cumulant formula \eqref{eq:monotoneMCrel} for monotone cumulants can be written as
\begin{equation}
	\varphi(a_1
	\cdots a_n) = \sum_{\pi\in \text{M}(n)} \frac{1}{|\pi|!} h_\pi(a_1,\ldots,a_n),
\end{equation}
where the sum runs over so-called monotone partitions, i.e., non-crossing partitions equipped with a total order on their blocks thereby refining the natural partial order of the blocks.
\end{rem}

\begin{rem}
	Cumulants linearize additive convolution. More precisely, let $(\A,\varphi)$ be a non-commutative probability space and assume that $a,b\in \A$ are free random variables. If $\mu_a$ and $\mu_b$ are distributions of $a$ respectively $b$, then we define \textit{free additive convolution of $a$ and $b$}, denoted by $\mu_a \boxplus \mu_b$, to be the distribution of their sum, $a+b$. For more details about free additive convolution, the reader is referred to standard textbook \cite{NSp}. 
	\par The relation with free cumulants is as follows: take free cumulants $\{k_n\}_{n\geq1}$  on $(\A,\varphi)$ and let $\mu$ be the distribution of a random variable $x \in \A$. The sequence of free cumulants of $\mu$ is the sequence of complex numbers $\{k_n(\mu)\}_{n\geq1}$ given by $k_n(\mu):= k_n(x,\ldots,x)$, for any $n\geq1$. The property of free cumulants implies that
	$$
	k_n(\mu_a\boxplus \mu_b) = k_n(\mu_a)  + k_n(\mu_b), \quad\mbox{for any }\,n\geq1.
	$$
	Analogously, if $a$ and $b$ are Boolean independent random variables, the Boolean cumulants of the \textit{Boolean additive convolution} $\mu_a\uplus\mu_b$ satisfy
	$$
	b_n(\mu_a\uplus\mu_b) =  b_n(\mu_a)  + b_n(\mu_b), \quad\mbox{for any }\,n\geq1. 
	$$
	Note that an analogous result for the monotone case is not available in general. However, monotone cumulants linearize monotone convolution when $a$ and $b$ have the same distribution. In other words, if $a$ and $b$ are monotone independent random variables such that $\mu_a = \mu_b$, then its \textit{monotone additive convolution} $\mu_a\blacktriangleright \mu_b$, defined to be the distribution of the sum $a+b$, satisfies 
	$$
	h_n(\mu_a\blacktriangleright \mu_b) = h_n(\mu_a) + h_n(\mu_b) = 2h_n(\mu_a),\quad\mbox{for any }n\geq1.
	$$
\end{rem}

\begin{rem}
	One can show that the definitions of additive convolutions only depend on the distributions of the random variables. Hence, it makes sense to consider the additive convolutions $\mu_1\boxplus\mu_2$, $\mu_1\uplus\mu_2$ and $\mu_1\blacktriangleright \mu_2$ of any pair of linear functionals $\mu_1,\mu_2:\mathbb{C}[X]\to\mathbb{C}$ such that $\mu_1(1) = \mu_2(1) = 1$.
\end{rem}


\subsection{Conditionally free independence}
\label{ssec:condfree}

Conditionally free independence, introduced by Bo\.zejko and Speicher in \cite{BS91}, aims to unify free, Boolean, and monotone independence within an augmented framework. This framework involves a non-commutative probability space $(\A, \varphi)$, complemented by an additional unital linear functional $\psi$ with $\psi(1_\A) = 1$. Such a triple is called a \textit{conditionally non-commutative probability space} and forms the basis for the exploration of conditionally free independence.

\begin{defi}[\cite{BS91}]
	Let $(\A,\varphi,\psi)$ be a conditionally non-commutative probability space and $\A_1, \ldots,$ $\A_k$ be unital subalgebras of $\A$. We say that $\A_1, \ldots, \A_k$ are \textit{conditionally free} (or simply \textit{c-free}) if the following holds: for $n \geq 2$, any sequence of indices $i_1,\ldots, i_n \in [k]$ such that $i_j\neq  i_{j+1}$ for $1 \leq j < n$, and elements $a_1 \in \A_{i_1} ,\ldots, a_n \in \A_{i_n}$ such that $\psi(a_j ) = 0$ for $j = 1, \ldots, n$, the following is satisfied:
	\begin{enumerate}[label=\roman*)]
		\item $\psi(a_1\cdots a_n) = 0$, i.e.~$\A_1,\ldots,\A_k$ are free with respect to $\psi$;
		\item $\varphi(a_1 \cdots a_n) = \varphi(a_1) \cdots \varphi(a_n)$. 
	\end{enumerate}
\end{defi}

In the work \cite{BLS96}, the authors introduced the notion of cumulants for c-free independence that unifies free and Boolean cumulants.

\begin{defi}
\label{def:cfreecumulants}
Let $(\A,\varphi,\psi)$ be a conditionally non-commutative probability space. Its \textit{c-free cumulants} form the family of multilinear functionals $\{k^{(c)}_n : \A^n \to \mathbb{C}\}_{n\geq1}$ recursively defined by
$$
	\varphi(a_1\cdots a_n) = \sum_{\pi\in \NC(n)}
	\left(\prod_{\substack{\pi_i\in\pi\\\pi_i\,\mathrm{outer}}} k^{(c)}_{|\pi_i|}(a_1,\ldots,a_n|\pi_i) \right)
	\left(\prod_{\substack{\pi_j\in\pi\\\pi_j\,\mathrm{inner}}} k'_{|\pi_j|}(a_1,\ldots,a_n|\pi_j) \right),
$$
for any $n \geq 1$ and $a_1,\ldots,a_n\in\A$, where $\{k'_n\}_{n\geq1}$ are the free cumulants of $(\A, \psi)$.
\end{defi}

Let $a \in \A$ be a random variable in a conditionally non-commutative probability space $(\A,\varphi,\psi)$. In this framework, the notion of distribution of $a$ extends to a pair of linear functionals $(\mu,\nu)$ on $\mathbb{C}[X]$ given by
$$	
	\mu(X^n) 
	= \varphi(a^n)\quad\mbox{and }\quad \nu(X^n) 
	= \psi(a^n),\quad\mbox{ for any }n\geq0.
$$
In addition, it makes sense to define the c-free cumulants of the pair $(\mu,\nu)$ as the sequence of complex numbers $\{k^{(c)}_n(\mu,\nu)\}_{n\geq1}$ given by $k^{(c)}_n(\mu,\nu) = k^{(c)}_n(a,\ldots,a)$, for any $n\geq1$. 

\begin{defi}
\label{def:cfreeaddconv}
 Let $(\A,\varphi,\psi)$ be a conditionally non-commutative probability space with c-free cumulants $\{k^{(c)}_n\}_{n\geq1}$. Given two pairs of distributions $(\mu_1,\nu_1)$ and $(\mu_2,\nu_2)$, the \textit{c-free additive convolution} is defined to be the pair 
\begin{equation}
\label{eq:cfreeaddconv}
	(\mu,\nu)= (\mu_1,\nu_1)\boxplus  (\mu_2,\nu_2)
\end{equation} 
such that $\nu = \nu_1 \boxplus \nu_2$, and $\mu$ is the distribution such that 
$$
	k^{(c)}_n(\mu,\nu) 
	= k^{(c)}_n(\mu_1,\nu_1)+ k^{(c)}_n(\mu_2,\nu_2),\quad\mbox{for any }n\geq1.
$$
\end{defi} 

Note that c-free additive convolution \eqref{eq:cfreeaddconv} is sometimes denoted $\mu_1\! \leftindex_{\nu_1} \boxplus_{\nu_2} \mu_2 := \mu$. In addition, it can be shown that if $a$ and $b$ are c-free random variables with distributions $(\mu_1,\nu_1)$ respectively $(\mu_2,\nu_2)$,  then the distribution of the sum $a+b$ is given precisely by $(\mu_1,\nu_1)\boxplus  (\mu_2,\nu_2)$.   


\subsection{Conditionally monotone independence}

Although the notion of c-free independence also includes the notion of monotone independence \cite{Franz06}, monotone cumulants are not as easy to identify in terms of c-free cumulants, as free and Boolean cumulants are. In order to overcome this difficulty, Hasebe \cite{Has} introduce an analogue of c-free independence such that it identifies monotone and Boolean independence.

\begin{defi}[{\cite[Def.~3.8]{Has}}]
	\label{def:cmonotoneind}
	Let $(\A,\varphi,\psi)$ be a conditionally non-commutative probability space. We say that subalgebras $\A_1, \ldots, \A_k$ of $\A$ are \textit{conditionally monotone} (or simply \textit{c-monotone}) if the following properties are satisfied:
\begin{enumerate}[label=\roman*)]
	\item $\A_1,\ldots,\A_k$ are monotone independent with respect to $\psi$;
	\item $\varphi(a_1\cdots a_n) = \varphi(a_1)\varphi(a_2\cdots a_n)$ for 
	$a_1\in \A_{i_1},\ldots,a_n\in \A_{i_n}$, $i_j\in [k]$ and $i_1>i_2;$
	\item $\varphi(a_1\cdots a_n) = \varphi(a_1\cdots a_{n-1})\varphi(a_n)$ for 
	$a_1\in \A_{i_1},\ldots,a_n\in \A_{i_n}$, $i_j\in [k]$ and $i_{n-1}< i_n;$
	\item $\varphi(a_1\cdots a_n) = \varphi(a_1\cdots a_{j-1})\big(\varphi(a_j)-\psi(a_j)\big)\varphi(a_{j+1}\cdots a_n) 
	+ \psi(a_j)\varphi(a_1\cdots a_{j-1}a_{j+1}\cdots a_n)$ for 
	$a_1\in \A_{i_1},\ldots,a_n\in \A_{i_n}$, $i_j\in [k]$, with $i_{j-1}<i_j> i_{j+1}$, with $2\leq j< n$.
\end{enumerate}
\end{defi}

It can be easily verified from the preceding definition that in the case where $\varphi = \psi$, c-monotone independence implies monotone independence. Similarly, when $\psi = 0$, c-monotone independence implies Boolean independence. Consequently, c-monotone independence serves as a unifying concept for both monotone and Boolean independence. Analogously, Hasebe introduced the following notion of c-monotone cumulants, encompassing both monotone and Boolean cumulants as specific instances.

\begin{defi}[{\cite[Prop.~4.5]{Has}}]
\label{prop:Hasebe}
Let $(\A,\varphi,\psi)$ be a conditionally non-commutative probability space. Its \textit{c-monotone cumulants} form the family of multilinear functionals $\{h_n^{(c)}:\A^n\to\mathbb{C}\}_{n\geq1}$ recursively defined by the formula
\begin{equation}
\label{eq:CMonMC}
	\varphi(a_1\cdots a_n) = \sum_{\pi\in \NC(n)} \frac{1}{t(\pi)!} 
	\left(\prod_{\substack{\pi_i\in\pi\\\pi_i\,\mathrm{outer}}} h_{|\pi_i|}^{(c)}(a_1,\ldots,a_n|\pi_i ) \right)
	\left(\prod_{\substack{\pi_j\in\pi\\\pi_j\,\mathrm{inner}}} h'_{|\pi_j|}(a_1,\ldots,a_n|\pi_j) \right),
\end{equation}
for any $n\geq1$ and $a_1,\ldots,a_n\in\A$, where $\{h'_n\}_{n\geq1}$ are the monotone cumulants of $(\A,\psi)$.
\end{defi} 

More intriguingly, Hasebe introduced c-monotone cumulants to linearize the c-monotone analogue of c-free additive convolution, especially when dealing with identically distributed random variables. The definition of this c-monotone analogue is as follows.

\begin{defi}[\cite{Has}]
\label{def:cmonconv}
Let $(\A,\varphi,\psi)$ be a conditionally non-commutative probability space with c-monotone cumulants $\{h_n^{(c)}\}_{n\geq1}$. Given two pairs of distributions $(\mu_1,\nu_1)$ and $(\mu_2,\nu_2)$, the \textit{c-monotone additive convolution} is defined to be the pair of distributions $(\mu,\nu) $  given by
$$
	(\mu,\nu) 
	:= (\mu_1\! \leftindex_{\delta_0}\boxplus_{\nu_2} \mu_2, \nu_1\blacktriangleright \nu_2),
$$
where $\delta_0$ stands for the distribution of $0\in\mathbb{C}\subseteq \A$.
\end{defi}
	
The c-monotone additive convolution is denoted  $(\mu_1,\nu_1)\blacktriangleright (\mu_2,\nu_2) := (\mu,\nu)$. Analogously to the c-free case,  it can be shown that if $a$ and $b$ are c-monotone independent random variables with distributions $(\mu_1,\nu_1)$ and $(\mu_2,\nu_2)$, respectively, then the distribution of $a+b$ is given precisely by $(\mu_1,\nu_1) \blacktriangleright (\mu_2,\nu_2)$.




\section{Shuffle-algebraic approach to non-commutative probability}
\label{sec:shuffle}

We describe now the Hopf-algebraic approach to non-commutative probability \cite{EFP1,EFP2,EFP3,EFP4}.


\subsection{Double tensor Hopf algebra construction} 

Let $(\A,\varphi)$ be a non-commutative probability space. We consider the \textit{double tensor Hopf algebra over $\A$} defined on the vector space
$$
	T(T_+(\A)):= \bigoplus_{n\geq0} T_+(\A)^{\otimes n},\qquad\mbox{ where }
	\quad 
	T_+(\A):= \bigoplus_{n\geq1} \A^{\otimes n}.
$$
For notational convenience, pure tensors in $T_+(\A)$ are written as words $w=a_1 \cdots a_n \in \A^{\otimes n}$ omiting the tensor symbol. On the other hand, pure tensors in $T(T_+(\A))$ are written using a bar $|$, i.e.~if $w_1,w_2,\cdots,w_m\in T_+(\A)$, then we write $w_1|w_2|\cdots | w_m \in T_+(\A)^{\otimes m}$.
\\\par It turns out that we can endow the space $T(T_+(\A))$ with a graded Hopf algebra structure in the following way. As an algebra, the associative non-commutative product on $T(T_+(\A))$ is given by concatenation using the bar $|$. More precisely, if $v = v_1|\cdots |v_n$ and $w = w_1|\cdots |w_m$ are elements in $T(T_+(\A))$, then its product its defined to be
$$ 
	v|w := v_1|\cdots|v_n|w_1|\cdots | w_m \in T(T_+(\A)).
$$
It is easy to see that the empty word, denoted by $\uno$ works as a unit for the product. 

In order to describe the coproduct, we need the following notation. Let $w=a_1\cdots a_n\in \A^{\otimes n}$. For each $S =\{s_1<\cdots <s_l\}\subseteq [n]$, we denote $w_S := a_{s_1}\cdots a_{s_l}$. Also, if $J_1,\ldots,J_r$ stand for the connected components of $[n]\backslash S$ ordered according to their minimal elements, then we define the coproduct by
\begin{equation}
\label{eq:coproduct}
	\Delta(w) := \sum_{S\subseteq[n]} w_S \otimes w_{J_1}|\cdots |w_{J_r}.
\end{equation}
The definition of $\Delta$ is then extended linearly and multiplicatively in order to define a map $\Delta: T(T_+(\A))\to T(T_+(\A))\otimes T(T_+(\A))$. The counit is simply defined to be the unital algebra morphism $\epsilon: T(T_+(\A))\to\mathbb{C}$ given by $\epsilon(\uno)=1$ and $\epsilon(w)=0$ if $w\in T_+(\A)$.
\\\par As a vector space, $T(T_+(\A))$ is graded in the sense that
$$
	T(T_+(\A)) = \bigoplus_{n\geq0} T(T_+(\A))_n,
$$
where
$$
	T(T_+(\A))_n := \bigoplus_{n_1+\cdots+n_k=n} 
		T_+(\A)^{\otimes n_1}\otimes \cdots \otimes T_+(\A)^{\otimes n_k}.
$$
Moreover, the product and the coproduct respect the grading. By noticing that it is \textit{connected} in the sense that $T(T_+(\A))_0\cong \mathbb{C}$, one concludes that $(T(T_+(\A)),|,\Delta)$ is a connected graded bialgebra and hence a Hopf algebra, called the \textit{double tensor Hopf algebra of $\A$}.
\\\par A fundamental property of the coproduct \eqref{eq:coproduct} is its splitting into two parts: $\Delta = \Delta_\prec + \Delta_\succ$, where
$$
	\Delta_\prec(w) 
	= \sum_{1\in S\subseteq [n]} w_S\otimes w_{J_1}|\cdots |w_{J_r},
$$
for any word $w=a_1\cdots a_n \in T_+(\A)$. The triple $(T(T_+(\A)),\Delta_\prec,\Delta_\succ)$ fulfils the axioms of \textit{unital unshuffle (or dendriform) bialgebra}. For further details, including the precise definition of unshuffle bialgebras, we refer the reader to \cite[Thm.~5]{EFP1}. For our purposes, it suffices to translate the axioms of unshuffle bialgebras to the algebraic dual of the double tensor Hopf algebra. More precisely, for linear maps $f,g\in \operatorname{Hom}(T(T_+(\A)),\mathbb{C})$, the coproduct in \eqref{eq:coproduct} permits to define the convolution product 
\begin{equation}
\label{eq:convproduct}
	f*g := m_\mathbb{C}\circ(f\otimes g)\circ \Delta,
\end{equation}
where $m_\mathbb{C}$ stands for the product on $\mathbb{C}$. Coassociativity of $\Delta$ implies that the convolution product \eqref{eq:convproduct} is associative, with the counit $\epsilon$ being its unit. The splitting of the coproduct $\Delta$ allows to define two additional products, called left and right half-shuffles
\begin{equation}
\label{eq:halfshuffles}
	f\prec g := m_\mathbb{C}\circ (f\otimes g)\circ \Delta_\prec,
	\qquad 
	f\succ g:= m_\mathbb{C}\circ (f\otimes g)\circ \Delta_\succ,
\end{equation}
for $f,g\neq\epsilon$, and defining for $f \neq \epsilon$
$$
	f\prec \epsilon= f=\epsilon\succ f, 
	\quad
	\epsilon\prec f = 0 = f\succ \epsilon.
$$
One can check that neither of the products \eqref{eq:halfshuffles} is associative. However, they satisfy the so-called \textit{shuffle axioms}
\begin{equation}
\label{eq:shuffle}
	(f\prec g)\prec h = f\prec(g*h),
	\quad 
	(f\succ g)\prec h = f\succ (g\prec h),
	\quad 
	f\succ(g\succ h) = (f*g)\succ h,
\end{equation}
for any $f,g,h\in \operatorname{Hom}(T(T_+(\A)),\mathbb{C})$. In general, any vector space $D$ endowed with half-shuffle products $\prec,\succ$ satisfying the relations \eqref{eq:shuffle}, where $a*b := a\prec b + a\succ b$, is called a non-commutative \textit{shuffle algebra}. In other words, we have:

\begin{prop}[\cite{EFP1}]
\label{prop:shufflealg}
$(\operatorname{Hom}(T(T_+(\A)),\mathbb{C}),\prec,\succ)$ is a unital non-commutative shuffle algebra.
\end{prop}

In $\operatorname{Hom}(T(T_+(\A)),\mathbb{C})$, we can identify a special group and Lie algebra as follows. Let $G$ be the set of \textit{characters on $T(T_+(\A))$}, i.e.~$\Phi\in G$ if $\Phi$ is a unital algebra morphism. On the other hand, let $\mathfrak{g}$ be the set of \textit{infinitesimal characters on $T(T_+(\A))$}, i.e.~$\alpha\in\mathfrak{g}$ if $\alpha(\uno)=0=\alpha(w_1|\cdots|w_k)$ for any $k\geq2$ and $w_1,\ldots,w_k\in T_+(\A)$. It is possible to show that $G$ is a group with respect to the convolution product \eqref{eq:convproduct}, and that $\mathfrak{g}$ is a Lie algebra with respect to the Lie bracket $[\alpha_1,\alpha_2] = \alpha_1*\alpha_2 - \alpha_2*\alpha_1$. In addition, there exists a bijection between $G$ and $\mathfrak{g}$ given by the exponential map with respect to the convolution product:
$$
	\mathfrak{g}\ni \alpha \mapsto \exp^*(\alpha)
	:=\sum_{n\geq0} \frac{\alpha^{*n}}{n!}\in G.
$$
In addition, the half-shuffle products $\prec$ and $\succ$ in $\operatorname{Hom}(T(T_+(\A)),\mathbb{C})$ allow to define for any $\alpha \in \mathfrak{g}$, the \textit{left} and \textit{right half-shuffle exponential maps}
$$
	\E_\prec(\alpha) 
	= \sum_{n\geq0} \alpha^{\prec n},\quad\quad \E_\succ(\alpha) 
	= \sum_{n\geq0} \alpha^{\succ n},
$$
where $\alpha^{\prec n} := \alpha \prec (\alpha^{\prec n-1})$ for $n\geq1$, and $\alpha^{\prec 0}:=\epsilon$; analogously for $\alpha^{\succ n}$. It turns out that the half-shuffle exponential maps also provide bijections between $\mathfrak{g}$ and $G$. The main result in \cite{EFP1} is stated in the following theorem.

\begin{theo}[\cite{EFP1,EFP2}]
\label{thm:mainEFP}
For $\Phi \in G$ a character, there exists a unique triple $(\kappa,\beta,\rho)$ of infinitesimal characters in $ \mathfrak{g}$ such that
\begin{equation}
\label{eq:mainEFP}
	\Phi = \exp^*(\rho) = \E_\prec(\kappa) = \E_\succ(\beta).
\end{equation}
In particular, the infinitesimal characters $\kappa$ and $\beta$ are the unique solutions to the fixed point equations
\begin{equation}
\label{eq:FixedEq}
	\Phi = \epsilon +\kappa\prec\Phi\quad\mbox{ and }\quad \Phi=\epsilon + \Phi\succ\beta,
\end{equation}
respectively. Conversely, we have that $\exp^*(\alpha)$, $\E_\prec(\alpha)$ and $\E_\succ(\alpha)$ are characters, for any $\alpha\in\mathfrak{g}$.
\end{theo}

Now, given a non-commutative probability space $(\A,\varphi)$, consider the double tensor Hopf algebra $T(T_+(\A))$ as well as  a character $\Phi$ on $T(T_+(\A))$ defined by the recipe $\Phi(w) = \varphi(a_1\cdots a_n)$, for any word $w=a_1\cdots a_n\in \A^{\otimes n}$. Note that in the argument of $\varphi$, we are considering the product on $\A$ of the random variables $a_1,\ldots,a_n\in \A$. Furthermore, given a family of multilinear functionals $\{f_n:\A^n\to\mathbb{C}\}_{n\ge 1}$, we define the infinitesimal character $\alpha \in \mathfrak{g}$ associated to the family $\{f_n\}_{n\geq1}$ by defining $\alpha(w) := f_n(a_1,\ldots,a_n)$, for any word $w=a_1\cdots a_n\in \A^{\otimes n}$. The next theorem shows that the relations between moments and cumulants in non-commutative probability can be described through the bijections between the group $G$ of characters and the Lie algebra $ \mathfrak{g}$ of infinitesimal characters of the double tensor Hopf algebra.

\begin{theo}[\cite{EFP1,EFP2}]
\label{thm:linkNCP}
Let $(\A,\varphi)$ be a non-commutative probability space and consider the double tensor Hopf algebra $T(T_+(\A))$. Let $\Phi$ be the character on $T(T_+(\A))$ extending $\varphi$ as described above, and let $(\kappa,\beta,\rho)$ be the unique triple of infinitesimal characters given in \cref{thm:mainEFP}. Then for any word $w=a_1\cdots a_n\in \A^{\otimes n}$, we have
$$
	\kappa(w) 
	= k_n(a_1,\ldots,a_n),\quad \beta(w) 
	= b_n(a_1,\ldots,a_n),\quad \rho(w) = h_n(a_1,\ldots,a_n).
$$
In other words, $\kappa$, $\beta$ and $\rho$ are the infinitesimal characters on $T(T_+(\A))$ associated to the free, Boolean, and monotone cumulants, respectively.
\end{theo}

The above theorem establishes that the moment-cumulant relations described in \cref{def:cumulants} are obtained from \eqref{eq:mainEFP} when evaluated in a word $w=a_1\cdots a_n\in T_+(\A)$. In conclusion, \cref{thm:linkNCP} provides an algebraic framework for moments and cumulants in non-commutative probability. 
\\\par One of the first applications of the shuffle algebraic framework consists in deriving combinatorial formulas between the different brands of cumulants, studied in detail by Arizmendi et al in \cite{AHLV}. Indeed, let $(\A,\varphi)$ be a non-commutative probability space. From the fixed point equations \eqref{eq:FixedEq} and the shuffle identities \eqref{eq:shuffle}, it is not difficult to deduce that 
\begin{equation}
\label{eq:FreeToBool}
	\beta = \Phi^{*-1}\succ\kappa \prec \Phi 
	\qquad\mbox{ and }\qquad 
	\kappa = \Phi\succ\beta\prec \Phi^{*-1}.
\end{equation}
The inverse $\Phi^{*-1} \in G$ is defined with respect to the shuffle product \eqref{eq:convproduct}. By evaluating in a word $w=a_1\cdots a_n\in T_+(\A)$, it is possible to obtain the known cumulant-cumulant formulas (\cite{AHLV})
$$
	b_n(a_1,\ldots,a_n) = \sum_{\pi\in\NCirr(n)} k_\pi(a_1,\ldots,a_n)
$$
and
$$  
	k_n(a_1,\ldots,a_n) = \sum_{\pi\in\NCirr(n)} (-1)^{|\pi|-1}b_\pi(a_1,\ldots,a_n).
$$
In general, we have that the evaluation of the shuffle adjoint action  of the group $G$ on the Lie algebra $\mathfrak{g}$
$$
	\theta_\Phi(\alpha) := \Phi^{*-1}\succ\alpha\prec \Phi,
$$ 
for any character $\Phi \in G$ and any infinitesimal character $\alpha \in \mathfrak{g}$, has the following combinatorial expression.

\begin{lemma}[{\cite[Lem.~15]{EFP4}}]
\label{lem:auxL15}
Let $\Phi\in G$ be a character on the double tensor Hopf algebra $T(T_+(\A))$ such that $\Phi = \E_\prec(\kappa) = \E_\succ(\beta)$, and let $\alpha\in\mathfrak{g}$ be an infinitesimal character. Then, for any word $w=a_1\cdots a_n \in T_+(\A)$ we have that
\begin{eqnarray*}
	\big(\Phi^{*-1}\succ  \alpha \prec \Phi\big)(w) 
	&=& \sum_{\substack{\pi\in \NCirr(n)\\ 1\in \pi_1}} \alpha(w_{\pi_1}) \prod_{\substack{\pi_j\in \pi\\ j\neq 1}} \kappa(w_{\pi_j}),\\ 
	\big(\Phi\succ \alpha \prec \Phi^{*-1}\big)(w) 
	&=& \sum_{\substack{\pi\in \NCirr(n)\\ 1\in \pi_1}} (-1)^{|\pi|-1}\alpha(w_{\pi_1}) \prod_{\substack{\pi_j\in \pi\\ j\neq 1}} \beta(w_{\pi_j}).
\end{eqnarray*} 
\end{lemma}

\begin{rem}
\label{rem:shuffleLogs}
Observe that \eqref{eq:mainEFP}, \eqref{eq:FixedEq} and \eqref{eq:FreeToBool} amount to defining half-shuffle logarithms 
\begin{equation}
\label{eq:shuffleLogs}
	\beta = \mathcal{L}_\succ (\Phi) := \Phi^{*-1}\succ (\Phi - \epsilon)
	\qquad 
	\text{and}
	\qquad
	 \kappa = \mathcal{L}_\prec (\Phi) := (\Phi - \epsilon)\prec \Phi^{*-1}.
\end{equation}
\end{rem}


\subsection{Pre-Lie Magnus expansion}

The cumulant-cumulant relations between monotone and free (and Boolean) cumulants are explained in the shuffle algebra framework via a pre-Lie analogue of the classical Magnus expansion \cite{Magnus1954}. More precisely, given a non-commutative probability space $(\A,\varphi)$, we have a non-associative product $\lhd:\mathfrak{g}\otimes \mathfrak{g}\to\mathfrak{g}$ on the Lie algebra of infinitesimal characters on $T(T_+(\A))$ by defining for any $\alpha,\gamma\in \mathfrak{g}$
\begin{equation}
\label{eq:preLieProd}
	\alpha \lhd \gamma := \alpha\prec\gamma - \gamma\succ\alpha.
\end{equation}
Instead of associativity, the product $\lhd$ satisfies the (right) \textit{pre-Lie identity}
\begin{equation}
\label{eq:preLie}
	(\alpha_1 \lhd \alpha_2) \lhd \alpha_3 - \alpha_1 \lhd (\alpha_2 \lhd \alpha_3) 
	= (\alpha_1 \lhd \alpha_3) \lhd \alpha_2 - \alpha_1 \lhd (\alpha_3 \lhd \alpha_2),
\end{equation}
for any $\alpha_1,\alpha_2,\alpha_3\in\mathfrak{g}$. In general, any vector space $L$ together with a product $\lhd$ satisfying \eqref{eq:preLie} is called a \textit{pre-Lie algebra} (see \cite{CP21,Manchon2011} for details and examples of pre-Lie algebras).

In the case of the pre-Lie algebra of infinitesimal characters $\mathfrak{g}$, the pre-Lie product allows us to define the \textit{pre-Lie Magnus operator \cite{AG81,EFM09}} $\Omega:\mathfrak{g}\to\mathfrak{g}$  as follows: for any $\alpha\in\mathfrak{g}$, define
\begin{equation}
\label{eq:Magnus}
	\Omega(\alpha) := \sum_{n\geq0} \frac{B_n}{n!} r_{\lhd\Omega(\alpha)}^{(n)} (\alpha),
\end{equation}
where in general, $r^{(0)}_{\lhd\gamma} (\alpha) = \alpha$,  and $r^{(n+1)}_{\lhd\gamma} (\alpha) = r^{(n)}_{\lhd\gamma} (\alpha) \lhd \gamma$, for $n \geq 0$, and $\{B_n\}_{n\geq0}$ are the Bernoulli numbers. The compositional inverse of the pre-Lie Magnus expansion is the operator $W:\mathfrak{g}\to\mathfrak{g}$ given for any $\alpha\in\mathfrak{g}$ by 
\begin{equation}
\label{eq:Wmap}
	W(\alpha) := \sum_{n\geq0} \frac{1}{(n+1)!} r^{(n)}_{\lhd\alpha}(\alpha).
\end{equation}

\begin{nota}
\label{not:exp}
We will use the following exponential notation for the operators $\Omega$ and $W$. For any $\alpha,\gamma\in\mathfrak{g}$, we set 
$$
	e^{r_{\lhd\gamma}}(\alpha) := \left(\sum_{n\geq0}\frac{1}{n!} r_{\lhd\gamma}^{(n)}\right)(\alpha).
$$ 
We can write for any $\alpha\in\mathfrak{g}$ 
$$
	W(\alpha) 
	= \left(\frac{e^{r_{\lhd\alpha}}-\id}{r_{\lhd\alpha}}\right)(\alpha),\qquad \Omega(\alpha) 
	= \left(\frac{r_{\lhd\Omega(\alpha)}}{e^{r_{\lhd\Omega(\alpha)}}-\id}\right)(\alpha).
$$
\end{nota}
It is shown in \cite{EFM09} that the pre-Lie Magnus expansion effectively describes the relation between the three logarithm maps in a shuffle algebra. More precisely, we have the following result. 

\begin{theo}[\cite{EFP2}]
	For a character $\Phi$ in $T(T_+(\A))$, the unique triple of infinitesimal characters $(\kappa,\beta,\rho)$ given in \cref{thm:mainEFP} satisfies the relations
    \begin{equation}
    	\beta = W(\rho),\quad \kappa = -W(-\rho)
    \end{equation}
and
	\begin{equation}
		\rho = \Omega(\beta) = -\Omega(-\kappa).
	\end{equation}
\end{theo}

In particular, the above theorem establishes that if $(\A,\varphi)$ is a non-commutative probability space and $\Phi$ is the character on $T(T_+(\A))$ extending $\varphi$, then the monotone cumulants of $(\A,\varphi)$ can be written in terms of the free and Boolean cumulants via the pre-Lie Magnus expansion. This approach has been recently applied in \cite{CEFPP} to provide combinatorial formulas, in terms of non-crossing partitions, that express monotone cumulants in terms of the free and Boolean ones:
\begin{align}
	h_n(a_1,\ldots,a_n) 
	&= \sum_{\pi\in\NCirr(n)}\omega(\pi)b_\pi(a_1,\ldots,a_n),\label{eq:auxMB1}\\ 
	h_n(a_1,\ldots,a_n) 
	&= \sum_{\pi\in\NCirr(n)}(-1)^{|\pi|-1}\omega(\pi)k_\pi(a_1,\ldots,a_n).\label{eq:auxMB2}
\end{align}
Here, for any $\pi\in\NCirr(n)$, the coefficient $\omega(\pi) := \sum_{k=1}^{|\pi|}\frac{(-1)^{k-1}}{k}\omega_k(\pi)$, with $\omega_k(\pi)$ being the number of surjective strictly order-preserving functions $f:\pi\to [k]$; we are considering $\pi$ as a poset according to its nesting structure. With the same approach, we can also recover the converse formulas from \cite{AHLV}:

\begin{align}
\label{eq:auxBM}
	b_n(a_1,\ldots,a_n) 
	&= \sum_{\pi\in\NCirr(n)}\frac{1}{t(\pi)!} h_\pi(a_1,\ldots,a_n),\\ 
	k_n(a_1,\ldots,a_n) 
	&= \sum_{\pi\in\NCirr(n)}\frac{(-1)^{|\pi|-1}}{t(\pi)!} h_\pi(a_1,\ldots,a_n).\label{eq:auxMB22}
\end{align}

The link between pre-Lie algebras and non-crossing partitions is explained by the following lemma that establishes a combinatorial formula for the right iteration of the pre-Lie product \eqref{eq:preLieProd}.

\begin{lemma}[\cite{CEFPP}]
\label{lem:iterated}
Let $\alpha_1,\ldots, \alpha_{n+1} \in \mathfrak{g} $ and $w = a_1\cdots a_m \in T_+(\A)$ be a word of length $|w| =m \geq n + 2$, $n > 0$. Then
\begin{equation}
\label{eq:iterated}
	((\cdots(\alpha_1\lhd \alpha_2)\lhd \cdots )\lhd \alpha_n)(w) 
	= \sum_{\substack{\pi\in \mathrm{M}_{n}^{\scriptscriptstyle\mathrm{irr}}(m)\\ \pi = (\pi_1,\ldots,\pi_n)}} 
	\prod_{j=1}^n \alpha_j(w_{\pi_j}),
\end{equation}
where $\mathrm{M}^{\scriptscriptstyle\mathrm{irr}}_n(m)$ stands for the set of irreducible monotone partitions on $[m]$ with exactly $n$ blocks.
\end{lemma}


\subsection{Natural products as characters}

Muraki showed in \cite{Mur} that there exist five natural products on non-commutative probability spaces. In \cite{EFP3}, the authors prove that the shuffle-algebraic framework effectively describes the free, Boolean and monotone products. 
\par In order to describe natural products in the shuffle picture, we consider the free product $\A := \A_1 \star \A_2$ of two non-commutative probability spaces, ($\A_1,\varphi_1)$ and $(\A_2,\varphi_2)$; identifying or not units, depending on whether we regard free, Boolean or monotone independence. Then, we consider $T(T_+(\A_1))$ and $T(T_+(\A_2))$, as well as $\Phi_1$ and $\Phi_2$, the characters extending $\varphi_1$ and $\varphi_2$, respectively. Moreover, we write $\Phi_1 = \E_\prec(\kappa_1) = \E_\succ(\beta_1)$ and $\Phi_2 = \E_\prec(\kappa_2) = \E_\succ(\beta_2)$.
\par The next step consists in considering the double tensor Hopf algebra $T(T_+(\A))$. Observe that we can identify $T(T(\A_1))$ and $T(T_+(\A_2))$ as  sub-Hopf algebras of $T(T_+(\A))$. Also notice that the linear forms $\Phi_1$, $\kappa_1$ and $\beta_1$ on $T(T_+(\A_1))$ can be extended to linear forms on $T(T_+(\A))$ by declaring $\Phi_1(w)=\kappa_1(w) = \beta_1(w)=0$ whenever $w$ contains a letter not in $\A_1$. Analogously, we can extend the linear forms $\Phi_2$, $\kappa_2$ and $\beta_2$ to linear forms on $T(T_+(\A))$.
\par Since the set of infinitesimal characters on $T(T_+(\A))$ forms a Lie algebra, the sum of infinitesimal characters is also an infinitesimal character. In particular, we may consider the infinitesimal characters $\kappa_1+\kappa_2$ and $\beta_1+\beta_2$ as well as the characters
\begin{equation}
\label{eq:ShuffleConv}
	\Phi_1\boxplus \Phi_2 
	:= \E_\prec(\kappa_1+\kappa_2),\qquad \Phi_1\uplus\Phi_2 
	:= \E_\succ(\beta_1+\beta_2).
\end{equation}
In particular, notice that $(\kappa_1+\kappa_2)(w) = 0$ whenever $w$ contains letters from both $\A_1$ and $\A_2$, which resembles the well-known result that free independence is characterized by the vanishing condition of mixed free cumulants \cite{NSp}. The following result establishes that the above characters are indeed extensions of the free and Boolean products of $\varphi_1$ and $\varphi_2$, and also provides the analog result for the monotone case.

\begin{theo}[\cite{EFP3}]
\label{thm:ShuffleConv}
Let $(\A_1,\varphi_1)$ and $(\A_2,\varphi_2)$ be two non-commutative probability spaces and consider the free product algebra $\A = \A_1\star \A_2$ (units are identified in the free case). If $\Phi_1$ and $\Phi_2$ are the corresponding characters on $T(T_+(\A))$ extending $\varphi_1$ respectively $\varphi_2$, then we have that
\begin{enumerate}[label=\roman*)]
	\item $\Phi_1\boxplus\Phi_2$ is the character extending the free product of $\varphi_1$ and $\varphi_2$.
	\item $\Phi_1\uplus\Phi_2$ is the character extending the Boolean product of $\varphi_1$ and $\varphi_2$.
	\item $\Phi_1*\Phi_2$ is the character extending the monotone product of $\varphi_1$ and $\varphi_2$.
\end{enumerate}
\end{theo}

\begin{rem}
\label{rmk:consideration}
Let $(\A,\varphi)$ be a non-commutative probability space and consider two random variables $a$ and $b$. Denoting $\mu_a$ and $\mu_b$ the distributions of $a$ respectively $b$, we can define the characters $\Phi_a$ and $\Phi_b$ on $T(T_+(\A))$ extending $\mu_a$ respectively $\mu_b$, by $\Phi_a(w) = \mu_a(X^n)$ if $w=a^{\otimes n}\in \A^{\otimes n}$ and $n\geq0$, and $\Phi(w) =0$ otherwise, and analogously for $\Phi_b$. From \cref{thm:ShuffleConv}, the linear form $\Phi_a\boxplus \Phi_b$ is the character on $T(T_+(\A))$ extending the free product of $\mu_a$ and $\mu_b$. In particular, if restricting $\Phi_a\boxplus\Phi_b$ to the sub-Hopf algebra $T(T_+(\A_{a+b}))$ where $\A_{a+b}$ is the unital algebra generated by $a+b$, then we have a character that extends the free additive convolution $\mu_a\boxplus\mu_b$. Thus, we will also refer to \textit{$\Phi_a\boxplus\Phi_b$ as the character extending $\mu_a\boxplus\mu_b$}. Analogous considerations follow for the cases of Boolean and monotone additive convolutions.
\end{rem}


\subsection{Conditionally free cumulants as infinitesimal characters}

The shuffle algebraic point of view for non-commutative probability includes the conditionally free analogue of cumulants and convolutions (\cite{EFP4}). More explicitly, let $(\A,\varphi,\psi)$ be a conditionally non-commutative probability space and consider $T(T_+(\A))$ as well as the characters $\Phi$ and $\Psi$ extending $\varphi$ respectively $\psi$. If in addition, we consider $\beta$ to be the infinitesimal character associated to the Boolean cumulants of $\varphi$, i.e.~$\Phi = \E_\succ(\beta)$, we can define the infinitesimal character $\mathsf{K} \in \mathfrak{g}$ given by
\begin{equation}
\label{eq:CFreeChar}
	\mathsf{K} = \Psi\succ \beta\prec \Psi^{*-1}.
\end{equation}

\begin{rem}
\label{rmk:transporting}
As mentioned in the \nameref{sec:intro}, the definition of $\mathsf{K}$ can be described as a way of transporting the Boolean cumulant form $\beta$ of $\Phi$ through the shuffle adjoint action of $\Psi^{*-1}$ in the space of infinitesimal characters, i.e.~
\begin{equation}
\label{eq:transporting}
	\beta = \theta_{\Psi}(\mathsf{K}).
\end{equation}
The action extends the shuffle relation $\beta'= \theta_\Psi(\kappa')$ between the Boolean and free cumulants of $\Psi$, by considering $\beta$ instead of the Boolean cumulant form of $\Psi$. Evaluating \eqref{eq:transporting} on a word $w=a_1\cdots a_n\in\A^{\otimes n}$, using \cref{lem:auxL15} and the Boolean moment-cumulant formula, we can prove that the infinitesimal character $\mathsf{K}$ is actually associated to the c-free cumulants of $(\A,\varphi,\psi)$.
\end{rem}

\begin{theo}[\cite{EFP4}]
	Let $(\A,\varphi,\psi)$ be a conditionally non-commutative probability space, and consider the infinitesimal character $\mathsf{K}$ on $T(T_+(\A))$ defined in \cref{eq:CFreeChar}. Then for any word $w=a_1\cdots a_n\in T_+(\A)$, we have
$$
		\mathsf{K}(w) = k_{n}^{(c)}(a_1,\ldots,a_n).
$$
In other words, $\mathsf{K}$ is the infinitesimal character  on $T(T_+(\A))$ associated to the conditionally free cumulants $\{k_n^{(c)}\}_{n\geq1}$ of $(\A,\varphi,\psi)$.
\end{theo}

Conditionally free additive convolution is described as a pair of characters in the following way. Given two pairs of characters, $(\Phi_1,\Psi_1)$ and $(\Phi_2,\Psi_2)$, defined on $T(T_+(\A))$, we write $\Psi_1 = \E_\prec(\kappa_1')$ and $\Psi_2 = \E_\prec(\kappa_2')$. Also, we denote by $\mathsf{K}_1$ and $\mathsf{K}_2$ the infinitesimal characters given by \eqref{eq:CFreeChar} associated to $(\Phi_1,\Psi_1)$ and $(\Phi_2,\Psi_2)$, respectively. Thus, we define the c-free convolution in the shuffle picture by
\begin{equation}
\label{eq:cfreeconv}
	(\Phi_1,\Psi_1) \boxplus (\Phi_2,\Psi_2) := (\Phi,\Psi),
\end{equation}
where 
\begin{equation}
\label{eq:psiconv}
	\Psi = \Psi_1\boxplus \Psi_2 = \E_\prec(\kappa_1'+ \kappa_2'),
\end{equation}
and $\Phi$ is the character given by
\begin{equation}
\label{eq:phiconv}
	\Phi = \E_\succ \big(\Psi^{*-1}\succ (\mathsf{K}_1+\mathsf{K}_2)\prec \Psi \big).
\end{equation}
In the language of non-commutative probability, Boolean cumulants of the distribution $\Phi$ are given by the infinitesimal character $\Psi^{*-1}\succ (\mathsf{K}_1+\mathsf{K}_2)\prec \Psi$, i.e.~the c-free additive convolution has Boolean cumulants that are obtained by transporting the sum of c-free cumulants by the free additive convolution $\Psi_1\boxplus \Psi_2$ using the shuffle adjoint action: $\mathcal{L}_\succ(\Phi) = \beta = \theta_{\Psi_1\boxplus\Psi_2}(\mathsf{K}_1+\mathsf{K}_2)$.

Before stating the next theorem, recall that the distribution of a random variable $a\in \A$ in a conditionally non-commutative probability space $(\A,\varphi,\psi)$ consists of a pair of linear functionals $(\mu,\nu)$. The distribution can be extended to a pair of characters $(\Phi,\Psi)$ on $T(T_+(\A))$ given by $\Phi(w) = \mu(X^n)$ and $\Psi(w) = \nu(X^n)$ if $w = a^{\otimes n}\in \A^{\otimes n}$, and $\Phi(w) = 0 = \Psi(w)$ if $w$ is a word containing a letter different from $a$. The next result states how the shuffle-algebraic c-free additive convolution actually extends the c-free additive convolution of two pairs of distributions, similarly as explained in \cref{rmk:consideration}.

\begin{theo}[\cite{EFP4}]
Let $(\A,\varphi,\psi)$ be a conditionally non-commutative probability space and consider two pairs of distributions $(\mu_1,\nu_1)$ and $(\mu_2,\nu_2)$, with their respective extensions $(\Phi_1,\Psi_1)$ and $(\Phi_2,\Psi_2)$ to two pairs of characters on $T(T_+(\A))$. Then $(\Phi_1,\Psi_1) \boxplus (\Phi_2,\Psi_2)$ is the pair of characters extending the c-free additive convolution $(\mu_1,\nu_1)\boxplus(\mu_2,\nu_2)$.
\end{theo}


\section{Conditionally monotone cumulants as an infinitesimal characters}
\label{sec:cmoninfchar}

The goal of this section is to present one of the principal contributions of the paper: describing c-monotone cumulants as an infinitesimal character on the double tensor Hopf algebra. Throughout this section, $(\A,\varphi,\psi)$ is a conditionally non-commutative probability space, $T(T_+(\A))$ is the double tensor Hopf algebra on $\A$, and $\Phi$ and $\Psi$ are the characters on $T(T_+(\A))$ which extend $\varphi$ and $\psi$, respectively. We also consider the triples of infinitesimal characters $(\kappa,\beta,\rho)$ and $(\kappa',\beta',\rho')$ given by \cref{thm:mainEFP} for $\Phi$ and $\Psi$, respectively.
\par The definition of c-free cumulants as an infinitesimal character described in \eqref{eq:CFreeChar} can be considered as a generalization of the half-shuffle adjoint action equations \eqref{eq:FreeToBool} that relate the free and Boolean cumulants. This observation prompt us the following definition.

\begin{defi}[c-monotone cumulant infinitesimal character]
\label{def:cmoninfchar}
For the pair of characters $(\Phi,\Psi)$, we define the \textit{conditionally monotone cumulant infinitesimal character of $(\Phi,\Psi)$} to be the infinitesimal character $\mathsf{P} = \mathsf{P}(\Phi,\Psi) \in \mathfrak{g}$ defined by the equation
\begin{equation}
	\label{eq:cmoncm}
	W_{\rho'}(\mathsf{P}) := \mathsf{P} +  \sum_{n > 0} \frac{1}{(n+1)!} r^{(n)}_{\lhd\rho'}(\mathsf{P}) = \beta.
\end{equation}
\end{defi}

In the preceding definition, it is crucial to observe the parallelism between the relation $\beta = W_{\rho'}(\mathsf{P})$ and the relation $\beta = \theta_{\Psi}(\mathsf{K})$ in \eqref{eq:transporting}. Consequently, the c-monotone cumulant infinitesimal character is defined as the element in the Lie algebra $\mathfrak{g}$ which is mapped by $W_{\rho'}$ to the Boolean cumulant form of $\Phi$. This bears a striking resemblance to our shuffle-algebraic interpretation of the definition of c-free cumulants, outlined in \cref{rmk:transporting}.

\begin{rem}
In the case that $\Phi = \Psi$, we have that $\rho = \rho'$ and thus $W_\rho(\mathsf{P}) = \beta$ implies that $\mathsf{P} = \rho$ since $W_\rho$ coincides with the inverse of the pre-Lie Magnus operator \eqref{eq:Wmap}. In other words, $\mathsf{P}$ coincides with the infinitesimal character associated to the monotone cumulants of $\varphi$. On the other hand, the case $\rho' = 0$ implies that $\mathsf{P}=\beta$, i.e.~$\mathsf{P}$ coincides with the infinitesimal character associated to the Boolean cumulants of $\varphi$.
\end{rem}

In the context of \cref{def:cmoninfchar}, we may consider the map $W_{\rho'} : \mathfrak{g}\to\mathfrak{g}$. In the case where $\Phi = \Psi$, we have that $W_{\rho'}$ is the inverse of the pre-Lie Magnus operator. For the general case, the following result establishes that the operator $W_{\rho'}$ is invertible. 

\begin{prop}
\label{eq:MagnusExt}
The inverse with respect to composition of the operator  $W_{\rho'} : \mathfrak{g}\to\mathfrak{g}$ given by $W_\rho'(\alpha) = \sum_{n\geq0} \frac{1}{(n+1)!} r_{\lhd \rho'}^{(n)}(\alpha)$ is the operator $\Omega_{\rho'}:\mathfrak{g}\to\mathfrak{g}$ defined by
\begin{equation}
\label{eq:Magnusext0}
	\Omega_{\rho'}(\alpha) := \left(\frac{r_{\lhd\rho'}}{e^{r_{\lhd\rho'}}-\id}\right)(\alpha) =  \sum_{n\geq0} \frac{B_n}{n!} r^{(n)}_{\lhd\rho'}(\alpha),\qquad\mbox{for any }\,\alpha\in\mathfrak{g}.
\end{equation}
\end{prop}

For a proof of the previous statement, the reader can follow for instance the proof of \cite[Cor.~3.24]{Reute}. Once more, in the case $\Phi = \Psi$, we have that $\Omega_{\rho'}$ coincides with the pre-Lie Magnus operator given in \eqref{eq:Magnus}.

\begin{ex}
Let us evaluate \cref{eq:cmoncm} in the univariate case for words of small lengths. Using \cref{lem:iterated} with $\alpha_1 = \mathsf{P}$ and $\alpha_2 = \alpha_3 = \cdots = \alpha_n = \rho'$ and setting $a^n := a^{\otimes n}$, we obtain 
\begin{align*}
    \beta(a) 
    &= W_{\rho'}(\mathsf{P})(a) = \mathsf{P}(a)\\	
    \beta(a^2) 
    &= W_{\rho'}(\mathsf{P})(a^2) \\
    &= \mathsf{P}(a^2) + \frac{1}{2}(\mathsf{P}\lhd \rho')(a^2) = \mathsf{P}(a^2)\\
    \beta(a^3) 
    &= W_{\rho'}(\mathsf{P})(a^3) \\
    &= \mathsf{P}(a^3) + \frac{1}{2}(\mathsf{P}\lhd \rho')(a^3) = \mathsf{P}(a^3) + \frac{1}{2}\mathsf{P}(a^2)\rho'(a)\\
    \beta(a^4) 
    &= W_{\rho'}(\mathsf{P})(a^4) \\
    & = \mathsf{P}(a^4) + \mathsf{P}(a^3)\rho'(a) + \frac{1}{2}\mathsf{P}(a^2)\rho'(a^2) + \frac{1}{3}\mathsf{P}(a^2)\rho'(a)^2.
\end{align*}

Now, using the Boolean moment-cumulant formula, we proceed to compute moments
$$
	\Phi(a_1\cdots a_n) 
	= \E_\succ(\beta)(a_1\cdots a_n) 
	= \sum_{\pi\in \operatorname{I}(n)} \beta_\pi(a_1 \cdots a_n)
$$ 
in terms of the linear maps $\mathsf{P}$ and $\rho'$:
\begin{align*}
 \Phi(a)
 &= \beta(a) = \mathsf{P}(a)\\
    \Phi(a^2) 
    &= \beta(a^2) + \beta(a)^2 \\
    &= \mathsf{P}(a^2) + \mathsf{P}(a)^2\\
   \Phi(a^3) 
   &= \beta(a^3) + 2\beta(a^2)\beta(a) + \beta(a)^3 \\
   &= \mathsf{P}(a^3) + 2\mathsf{P}(a^2)\mathsf{P}(a) + \frac{1}{2}\mathsf{P}(a^2)\rho'(a) + \mathsf{P}(a)^3\\
   \Phi(a^4) 
   &= \beta(a^4) + \beta(a^2)\beta(a^2) + 2\beta(a^3)\beta(a) +  3\beta(a^2)\beta(a)^2 + \beta(a)^4 \\
   &= \mathsf{P}(a^4) + \mathsf{P}(a^2)^2 +2\mathsf{P}(a^3)\mathsf{P}(a) + \mathsf{P}(a^3)\rho'(a) + 3\mathsf{P}(a^2)\mathsf{P}(a)^2 \\
   &\quad +\frac{1}{3}\mathsf{P}(a^2)\rho'(a)^2+\frac{1}{2}\mathsf{P}(a^2)\rho'(a^2)+  \mathsf{P}(a^2)\mathsf{P}(a)\rho'(a) + \mathsf{P}(a)^4,
\end{align*}
and observe that these computations match those given in Example 4.8 in \cite{Has}.
\end{ex}

In order to prove that the infinitesimal character $\mathsf{P}$ is indeed the associated linear form to the c-monotone cumulants, we now prove the following result which is analog to Lemma 15 in \cite{EFP4}.

\begin{lemma}
\label{lem:Wrho}
For any word $w=a_1\cdots a_n \in \A^{\otimes n}$ we have
\begin{equation}
\label{eq:lemma}
	W_{\rho'}(\mathsf{P})(w) 
	= \sum_{\substack{\pi\in \mathrm{M}^{\scriptscriptstyle\mathrm{irr}}(n) \\ 1\in \pi_1 \in \pi}}  \frac{1}{|\pi|!} \mathsf{P}(w_{\pi_1})
	\prod_{\substack{\pi_j\in \pi\\ j \neq 1}} \rho'(w_{\pi_j}).
\end{equation}
\end{lemma}

\begin{proof}
Recall \cref{eq:cmoncm}. By using \eqref{eq:iterated}, we have that
$$
 	r^{(m)}_{\lhd\rho'}(\mathsf{P})(w) 
	= \sum_{\substack{\pi\in \mathrm{M}^{\scriptscriptstyle\mathrm{irr}}_{m+1}(n)\\ \pi = (\pi_1,\ldots,\pi_{m+1})}} 
	\mathsf{P}(w_{\pi_1}) \prod_{j=2}^{m+1} \rho'(w_{\pi_j}),\quad\forall\,m\geq0,
$$
where $\pi_1$ is the unique outer block of the irreducible monotone partition $\pi$ which contains $1$ and $n$. We conclude then that
$$
	W_{\rho'}(\mathsf{P}) 
	= \sum_{m\geq0} \frac{1}{(m+1)!} r^{(m)}_{\lhd\rho'}(\mathsf{P}) 
	=  \sum_{\substack{\pi\in \mathrm{M}^{\scriptscriptstyle\mathrm{irr}}(n)\\1\in \pi_1}}  \frac{1}{|\pi|!} \mathsf{P}(w_{\pi_1})
	\prod_{\substack{\pi_j \in \pi \\ j \neq 1}} \rho'(w_{\pi_j}).
$$
\end{proof}

\begin{theo}
\label{prop:46P}
The linear form $\mathsf{P}$ satisfying \eqref{eq:cmoncm} is the infinitesimal character associated to the conditionally monotone cumulant functionals $\{h_n^{(c)} : \A^{n} \to \mathbb{C}\}_{n > 0}$ of $(\A,\varphi,\psi)$, i.e.~
$$
	\mathsf{P}(w) = h_n^{(c)}(a_1,\ldots,a_n),
$$
for any word $w= a_1\cdots a_n \in T_+(\A).$
\end{theo}

\begin{proof}
We will show that the evaluation of $\mathsf{P}$ on a word $w\in \A^{\otimes n}$ satisfies the moment-cumulant relation in \cref{prop:Hasebe}. Indeed, by definition of $\mathsf{P}$ and the right half-shuffle fixed point equation \eqref{eq:FixedEq}, we have that 
$$
	\Phi = \E_\succ(\beta) = \E_{\succ}(W_{\rho'}(\mathsf{P})) = \epsilon + \Phi\succ W_{\rho'}(\mathsf{P}).
$$ 
Then, for a word $w= a_1\cdots a_n \in \A^{\otimes n}$ we compute
\begin{eqnarray*}
	\Phi(w) 
	&=& \sum_{k=2}^n W_{\rho'}(\mathsf{P})(a_1\cdots a_{k-1}) \Phi(a_k\cdots a_n)\\ 
	&=&  \sum_{k=2}^n\left(\sum_{\substack{\pi^{(1)} \in \mathrm{M}^{\scriptscriptstyle\mathrm{irr}}(k-1)\\ 1 \in \pi^{(1)}_1}} 
	\frac{1}{|\pi^{(1)}|!} \mathsf{P}((a_1\cdots a_{k-1})_{\pi_1})
	\prod_{\substack{\pi_j \in \pi^{(1)}\\ j \neq 1}} \rho'((a_1\cdots a_{k-1})_{\pi_j}) \right) \Phi(a_k\cdots a_n),
\end{eqnarray*}
where we used that $W_{\rho'}(\mathsf{P})(\uno)=0$ in the first equality, and \cref{lem:Wrho} in the second equality. By induction on $\Phi(a_k\cdots a_n)$, we will obtain a sequence of irreducible non-crossing partitions $(\pi^{(1)},\ldots,\pi^{(s)})$ whose unique outer blocks index the subwords that are evaluated on $\mathsf{P}$. These irreducible non-crossing partitions are exactly the irreducible components of a non-crossing partition $\pi\in \NC(n)$. Also, by \cref{rmk:monotonepartitions}, we can re-index the sums as  sums over non-crossing partitions and replace $\frac{1}{|\pi|!}$ by $\frac{1}{t(\pi)!}$. Finally, by recalling that if $\pi$ is a non-irreducible non-crossing partition with irreducible components $\pi^{(1)},\ldots,\pi^{(s)}$, then $t(\pi)!  = t(\pi^{(1)})!\cdots t(\pi^{(s)})!$, we conclude that
$$
	\Phi(w) = \sum_{\pi\in\NC(n)} \frac{1}{t(\pi)!}\left( 
	\prod_{\substack{\pi_i\in\pi\\ \pi_i\,\mathrm{outer}}} \mathsf{P}(w_{\pi_i}) \right)\left( 
	\prod_{\substack{\pi_j\in\pi\\ \pi_j\,\mathrm{inner}}} \rho'(w_{\pi_j}) \right),
$$
as we wanted to show.
\end{proof}


\section{Conditionally monotone convolution via shuffle algebra}
\label{sec:condmonoshuffle}

After successfully establishing c-monotone cumulants within the shuffle-algebraic framework, the next objective is to formulate the c-monotone additive convolution as an operation applied to pairs of characters within the double tensor Hopf algebra $T(T_+(\A)).$
\par Recall that c-free cumulants of a pair of characters $(\Phi,\Psi)$ can be considered in the shuffle-algebraic picture by generalizing the half-shuffle equations \eqref{eq:FreeToBool} for $\Psi$ in such a way that $\mathsf{K}$ is obtained at the level of Lie algebra by transporting the Boolean cumulant  $\beta$ of $\Phi$ by $\Psi$, as described in \eqref{eq:CFreeChar}. In the same way, c-free additive convolution of $(\Phi_1,\Psi_1)$ and $(\Phi_2,\Psi_2)$ in the shuffle picture can be described by considering the Boolean cumulants of $\Psi = \Psi_1\boxplus \Psi_2$ but using $\mathsf{K}_1 + \mathsf{K}_2$ in \eqref{eq:FreeToBool} instead of the free cumulants of $\Psi$, so that we obtain the definition of $\Phi$ in \eqref{eq:phiconv}. In essence, the concepts of c-free cumulants and c-free convolution can be viewed as extensions of the free case. This extension is achieved by substituting the Boolean and free cumulants, $\beta'$ and $\kappa'$, of $\Psi$ with the Boolean cumulants $\beta$ of $\Phi$ respectively the linear form $\mathsf{K}$. Adopting this perspective allows us to recover the free case when setting $\Phi$ equal to $\Psi.$

Now, our aim is to formulate the additive convolution within the c-monotone case. The definition of c-monotone cumulants, as expressed through the infinitesimal character in \eqref{eq:cmoncm}, mirrors the approach taken in the c-free case. This involves transporting the Boolean cumulant $\beta$ associated with $\Phi$ via $\Psi$, utilizing an extension of the pre-Lie Magnus expansion. This extension establishes a connection between the Boolean and monotone cumulants of $\Psi$. To proceed with defining c-monotone additive convolution and drawing an analogy with the c-free case, our initial step is to delineate the Boolean cumulants of the monotone additive convolution $\Psi_1*\Psi_2.$

\begin{prop}
Let $\Psi_1$ and $\Psi_2$ be two characters on $T(T_+(\A))$ such that $\Psi_i = \E_\succ(\beta'_i)$, $i=1,2$. Then 
\begin{equation}
\label{eq:BCH}
	\Psi_1*\Psi_2 = \E_\succ(\beta'_1\#\beta'_2),
\end{equation}
where the infinitesimal character $\beta'_1\#\beta'_2:=\beta'_2 + \Psi_2^{*-1}\succ \beta'_1 \prec \Psi_2.$
\end{prop}

\begin{proof}
The result follows from the pre-Lie Magnus expansion and its relation with the Baker--Campbell--Hausdorff formula \cite{AG81}:
\begin{eqnarray*}
  \Psi_1*\Psi_2&=& \E_\succ(\beta'_1)*\E_\succ(\beta'_2)\\  
    &=& \exp^*(\Omega(\beta'_1))*\exp^*(\Omega(\beta'_2)) \\ 
    &=& \exp^*(\operatorname{BCH}(\Omega(\beta'_1),\Omega(\beta'_2))) \\ 
    &=& \exp^*(\Omega(\beta'_1 \# \beta'_2)) \\ 
    &=& \E_\succ (\beta'_1 \#\beta'_2).
\end{eqnarray*}
We refer the reader to  \cite[Sec.~4]{EFP2}, \cite[Sec.~1.1]{Manchon2011} and \cite[Sec.~6.6]{CP21} for details of the previous computation.
\end{proof}

Following the idea of replacing $\beta'$ by $\beta = \mathcal{L}_\succ(\Phi)$, we propose the following definition.

\begin{defi}[Conditionally monotone additive convolution]
\label{def:cmonconvC}
Let $(\Phi_1,\Psi_1)$ and $(\Phi_2,\Psi_2)$ be two pairs of characters on the double tensor Hopf algebra $T(T_+(\A))$. We define the \textit{conditionally monotone additive convolution of $(\Phi_1,\Psi_1)$ and $(\Phi_2,\Psi_2)$} as the pair of characters $(\Phi,\Psi)$ given by
\begin{eqnarray*}
        \Psi &=& \Psi_1*\Psi_2\\
        \Phi &=& \E_\succ\big(\beta_2 + \Psi_2^{*-1}\succ \beta_1\prec \Psi_2 \big),
\end{eqnarray*}
where $\Phi_1 = \E_\succ(\beta_1)$ and $\Phi_2 = \E_\succ(\beta_2)$. The conditionally monotone additive convolution of $(\Phi_1,\Psi_1)$ and $(\Phi_2,\Psi_2)$ is denoted by $(\Phi_1,\Psi_1) * (\Phi_2,\Psi_2)$.
\end{defi}

The previous definition induces an operation on the set of pair of characters on $T(T_+(\A))$ that is indeed associative.

\begin{prop}
The c-monotone additive convolution is an associative product on the set of pairs of characters of $T(T_+(\A))$.
\end{prop}

\begin{proof}
Let $(\Phi_i,\Psi_i)$, $i=1,2,3$, be pairs of characters on $T(T_+(\A))$, and we write $\Phi_i = \E_\succ(\beta_i)$. By definition, we have that
$$
	(\Phi_1,\Psi_1)*(\Phi_2,\Psi_2) = \big(\E_\succ(\beta) ,\Psi_1*\Psi_2\big),
$$
with $\beta:= \beta_2 + \Psi_2^{*-1}\succ\beta_1\prec \Psi_2$. Using the fact that the convolution product $*$ is associative on the set of linear functionals on $T(T_+(\A))$, we have
$$
	\big(\E_\succ(\beta),\Psi_1*\Psi_2\big) * (\Phi_3,\Psi_3)
	=\big(\E_\succ(\beta_3 + \Psi_3^{*-1}\succ \beta \prec \Psi_3 ) ,\Psi_1*\Psi_2*\Psi_3\big),
$$
where 
\begin{eqnarray*}
	\beta_3 + \Psi_3^{*-1}\succ \beta \prec \Psi_3  
	&=& \beta_3 + \Psi_3^{*-1}\succ (\beta_2 + \Psi_2^{*-1}\succ\beta_1\prec \Psi_2)\prec \Psi_3 \\ 
	&=&  \beta_3 + \Psi_3^{*-1}\succ\beta_2\prec \Psi_3 + (\Psi_2* \Psi_3)^{*-1}\succ\beta_1\prec (\Psi_2*\Psi_3).
\end{eqnarray*}
We used the shuffle identities \eqref{eq:shuffle} in the last equality above. On the other hand, we have that
$$
	(\Phi_2,\Psi_2)*(\Phi_3,\Psi_3) = \big(\E_\succ(\tilde\beta) ,\Psi_2*\Psi_3\big),
$$
with $\tilde\beta:= \beta_3 + \Psi_3^{*-1}\succ\beta_2\prec \Psi_3$. Again, by associativity of the convolution product, we get
$$
	(\Phi_1,\Psi_1)* \big(\E_\succ(\tilde\beta), \Psi_2*\Psi_3 \big) =
	 \big(\E_\succ(\tilde\beta + (\Psi_2*\Psi_3)^{*-1}\succ\beta_1 \prec (\Psi_2*\Psi_3) ) , \Psi_1*\Psi_2*\Psi_3\big).
$$
We conclude the proof by noticing that
\begin{eqnarray*}
	\beta_3 + \Psi_3^{*-1}\succ \beta \prec \Psi_3 
	&=& \beta_3 + \Psi_3^{*-1}\succ\beta_2\prec \Psi_3 + (\Psi_2* \Psi_3)^{*-1}\succ\beta_1\prec (\Psi_2*\Psi_3)\\ 
	&=& \tilde\beta + (\Psi_2*\Psi_3)^{*-1}\succ\beta_1 \prec (\Psi_2*\Psi_3).
\end{eqnarray*}
\end{proof}

The next proposition establishes that the associative product given in \cref{def:cmonconvC} actually extends the c-monotone additive convolution of distributions in \cite{Has}. Together with the previous proposition, the next result implies that the c-monotone additive convolution is an associative product in the space of pairs of distributions of random variables.

\begin{theo}
\label{prop:cmonconvC}
Let $(\A,\varphi,\psi)$ be a conditionally non-commutative probability space and consider two pairs of distributions, $(\mu_1,\nu_1)$ and $(\mu_2,\nu_2)$, with their respective extensions, $(\Phi_1,\Psi_1)$ and $(\Phi_2,\Psi_2)$, to two pairs of characters on $T(T_+(\A))$. Then $(\Phi_1,\Psi_1) * (\Phi_2,\Psi_2)$ is the pair of characters extending the c-monotone additive convolution $(\mu_1,\nu_1)\blacktriangleright(\mu_2,\nu_2)$.
\end{theo}

\begin{proof}
Let $(\Phi,\Psi) = (\Phi_1,\Psi_1)*(\Phi_2,\Psi_2)$. By definition, we have that $\Psi = \Psi_1*\Psi_2$, thus by \cref{thm:ShuffleConv}, $\Psi$ is the character extending the monotone convolution $\nu_1 \blacktriangleright \nu_2$. Now, we will prove that $\Phi$ satisfies the extension of \cref{def:cmonconv} to the case of characters, i.e.~$\Phi$ is the first component of 
$$
	 (\Phi,\Psi_0):=(\Phi_1,\epsilon)\boxplus (\Phi_2,\Psi_2),
$$
where the counit map $\epsilon$ is the character on $T(T_+(\A))$ extending the distribution of the random variable $0\in \A$. Indeed, first notice that $\epsilon = \E_\prec({\bf{0}})$, where ${\bf{0}}$ stands for the zero linear form on $T(T_+(\A))$. Then, if we write $\Psi_2 = \E_\prec(\kappa_2')$, by \eqref{eq:psiconv} we have that 
$$
	\Psi_0 = \E_\prec({\bf{0}} + \kappa_2') = \E_\prec(\kappa_2') = \Psi_2.
$$
On the other hand, the c-free cumulants of $(\Phi_1,\epsilon)$ and $(\Phi_2,\Psi_2)$ are given by 
\begin{equation}
	\mathsf{K}_1 
	= \epsilon\succ \beta_1 \prec \epsilon^{*-1} 
	= \beta_1,\qquad \mathsf{K}_2= \Psi_2\succ \beta_2\prec \Psi_2^{*-1}
\end{equation}
respectively, where we write $\Phi_i = \E_\succ(\beta_i)$ for $i=1,2$. Thus, from \eqref{eq:phiconv} we obtain that
\begin{eqnarray*}
	\Phi 
	&=& \E_\succ\big(\Psi_2^{*-1} \succ( \mathsf{K}_1+ \mathsf{K}_2)\prec \Psi_2 \big) \\ 
	&=&  \E_\succ\big(\Psi_2^{*-1} \succ \beta_1 \prec \Psi_2 
		+ \Psi_2^{*-1} \succ (\Psi_2\succ \beta_2\prec \Psi_2^{*-1}) \prec \Psi_2 \big) \\ 
	&=& \E_\succ\big(\Psi_2^{*-1} \succ \beta_1 \prec \Psi_2 
		+  (\Psi_2^{*-1} *\Psi_2) \succ \beta_2\prec (\Psi_2^{*-1}* \Psi_2) \big) \\ 
	&=& \E_\succ\big( \Psi_2^{*-1}\succ \beta_1\prec \Psi_2 
		+ \epsilon\succ\beta_2 \prec\epsilon\big)\\ 
	&=& \E_\succ\big( \Psi_2^{*-1}\succ \beta_1\prec \Psi_2 + \beta_2 \big),
\end{eqnarray*}
where we use the shuffle identities \eqref{eq:shuffle} in the third equality. We finish by noticing that the right-hand side on the last equality is precisely the definition of $\Phi$ in \cref{def:cmonconvC}.
\end{proof}

It has been shown in \cite{Has} that the c-monotone additive convolution extends the monotone and Boolean convolution. More precisely, for any distributions $\mu$ and $\nu$, and if $\delta_0$ stands for the distribution of $0\in \A$, we have that 
\begin{eqnarray*}
    (\mu,\mu)\blacktriangleright (\nu,\nu) 
    &=& (\mu\blacktriangleright \nu, \mu\blacktriangleright \nu),\\
    (\mu,\delta_0)\blacktriangleright (\nu,\delta_0) 
    &=& (\mu\uplus\nu,\delta_0).
\end{eqnarray*}
It is also shown in \cite{Has} that the c-monotone additive convolution extends the so-called \textit{orthogonal convolution} of Lenczewski, denoted by $\mu\vdash\nu$ in \cite{Len}.  One of the properties of the orthogonal convolution is that it decomposes the monotone convolution in terms of the Boolean convolution \cite[Cor.~6.6]{Len}:
\begin{equation}
	\label{eq:orthogonalproperty}
	\mu\blacktriangleright \nu = (\mu\vdash \nu)\uplus\nu.
\end{equation}
In terms of c-monotone additive convolution, the results of \cite{Has} show that
\begin{equation*}  
	(\mu,\lambda)\blacktriangleright (\delta_0,\nu) 
 	= (\mu\vdash \nu, \lambda \blacktriangleright\nu),\quad\mbox{ for any distributions $\mu,\nu,\lambda$}.
\end{equation*}

The definition of the Boolean additive convolution in \eqref{eq:ShuffleConv} and \eqref{eq:BCH} suggests a decomposition of the convolution product in terms of a sum of two Boolean cumulant forms. This implies the manner in which orthogonal convolution is encompassed in the shuffle-algebraic approach to non-commutative probability.
  

\begin{prop}
\label{prop:orthogonalShuffle}
Let $(\A,\varphi)$ be a non-commutative probability space and consider two distributions $\mu$ and $\nu$, with respective extensions $\Phi$ and $\Psi$ to two characters on $T(T_+(\A))$. Also, assume that $\Phi = \E_\succ(\beta)$. Then the linear form $\Phi\vdash\Psi$ given by
\begin{equation}
\label{eq:orthogonal}
    \Phi \vdash \Psi = \E_\succ ( \Psi^{*-1} \succ \beta \prec \Psi) 
\end{equation}
is the character on $T(T_+(\A))$ extending the orthogonal convolution $\mu\vdash\nu$. Moreover, if we assume that $\Psi = \exp^*(\rho')$, then
\begin{equation}
\label{eq:orthogonal2}
\Phi \vdash \Psi = \E_\succ(e^{r_{\lhd \rho'}}(\beta)).
\end{equation}
\end{prop}

\begin{proof}
The proof of \eqref{eq:orthogonal} follows from \eqref{eq:BCH} together with \eqref{eq:orthogonalproperty} and \cref{thm:ShuffleConv}. Regarding \eqref{eq:orthogonal2}, define the half-shuffle right- and left-multiplication operators $r_{\prec \rho'}(\beta) := \beta\prec\rho'$ and $\ell_{\rho'\succ }(\beta) := \rho'\succ\beta$, as well as the corresponding iterations $r_{\prec \rho'}^{(n)}(\beta)$ and $\ell_{\rho'\succ }^{(n)}(\beta)$, for $n\geq0$. The shuffle identities imply that for any $n\geq 1$
$$
	r^{(n)}_{\prec \rho'} (\beta) 
	= \beta\prec (\rho')^{*n},\qquad \ell_{\rho'\succ}^{(n)}(\beta) 
	= (\rho')^{*n}\succ\beta.
$$
Furthermore, using the shuffle identities \eqref{eq:shuffle}, we have that the operators $r_{\prec \rho'}^{(n)}$ and $\ell_{\rho'\succ }^{(n)}$ commute, for any $n\geq1$. This implies that the operators $e^{r_{\prec \rho'}}$ and $e^{\ell_{-\rho'\succ}}$, defined in a similar fashion as in \cref{not:exp}, commute. Hence, since the pre-Lie product $\beta \lhd \rho'= r_{\lhd \rho'}(\beta) = (r_{\prec \rho'} - \ell_{\rho'\succ })(\beta)$, we obtain
\begin{eqnarray*}
	e^{r_{\lhd \rho'}}(\beta) &=& e^{r_{\prec \rho'}}e^{\ell_{-\rho'\succ}}(\beta)\\ 
	&=& \exp^{*}(\rho')\succ\beta\prec \exp^{*}(\rho')\\ 
	&=& \Psi^{*-1}\succ\beta\prec \Psi\\
	&=& \mathcal{L}_\succ(\Phi \vdash \Psi).
\end{eqnarray*}
\end{proof}

\begin{ex} 
Let $\Phi,\Psi$ and $\Lambda$ be characters on the double tensor Hopf algebra $T(T_+(\A))$. Using \cref{def:cmonconvC}, we can directly deduce Hasebe's particular cases of c-monotone additive convolutions in the shuffle-algebraic picture.

        i) Consider the case $(\tilde\Phi,\tilde\Psi) := (\Phi,\Phi)* (\Psi,\Psi)$.  By definition, it is clear that $\tilde\Psi = \Phi*\Psi$. 
        Moreover, if $\Phi = \E_\succ(\beta)$ and $\Psi = \E_\succ(\beta')$, using \eqref{eq:BCH} we obtain
        \begin{equation*}
        		\Phi*\Psi = \E_\succ(\beta\#\beta') 
		= \E_\succ(\beta' + \Psi^{*-1} \succ \beta\prec \Psi) 
		= \tilde\Phi,
        \end{equation*}
        where in the last equality, we used the definition of the first component of the c-monotone additive convolution.  We then conclude that
        $$
        		(\Phi,\Phi)* (\Psi,\Psi) = (\Phi*\Psi,\Phi*\Psi)
	$$
	is the monotone additive convolution.
        
        \par ii) Now consider $(\tilde\Phi,\tilde\Psi) := (\Phi,\epsilon)* (\Psi,\epsilon)$. 
        It readily follows that $\tilde\Psi = \epsilon*\epsilon = \epsilon$. In addition, 
        if $\Phi = \E_\succ(\beta)$ and $\Psi = \E_\succ(\beta')$, then
        $$
        		\tilde\Phi 
		= \E_\succ( \beta' + \epsilon^{*-1}\succ\beta \prec\epsilon) 
		= \E_\succ(\beta + \beta').
	$$
        In other words, we have that $\tilde\Phi$ is the Boolean additive convolution $\Phi \uplus \Psi$, and hence 
        $$
        		(\Phi,\epsilon)* (\Psi,\epsilon) = (\Phi\uplus\Psi,\epsilon).
        $$
        \par iii) Finally, consider the case $(\tilde\Phi,\tilde\Psi) := (\Phi,\Lambda)* (\epsilon,\Psi)$. 
        Clearly $\tilde\Psi = \Lambda *\Psi$ and, if $\Phi = \E_\succ(\beta)$, we have that 
        $$
        		\tilde\Phi = \E_\succ(\Psi^{*-1}\succ \beta\prec\Psi),
	$$
	since the $\epsilon = \E_\succ({\bf{0}})$. We note that $\tilde\Phi$ is precisely the orthogonal convolution 
	of $\Phi$ and $\Psi$ in \eqref{eq:orthogonal}. Thus
        $$
        		(\Phi,\Lambda) * (\epsilon,\Psi) = (\Phi\vdash \Psi, \Lambda *\Psi).
	$$
\end{ex}

While \cref{prop:cmonconvC} suggests that existing findings regarding c-monotone convolution can be expressed in terms of Hopf algebra characters, it is important to highlight that these outcomes can be directly derived from the shuffle representation. To illustrate, consider the power additivity of c-monotone cumulants as a case in point.


\begin{prop}[{\cite[Thm.~4.4]{Has}}]
Let $\Phi$ and $\Psi$ be two characters on $T(T_+(\A))$ and $\mathsf{P} := \mathsf{P}(\Phi,\Psi)$ be the c-monotone cumulant infinitesimal character of $(\Phi,\Psi)$. Then for any $N\geq1$, we have that $\mathsf{P}\left( (\Phi,\Psi)^{* N})\right) = N\mathsf{P}(\Phi,\Psi)$.
\end{prop}

\begin{proof}
We prove the case $N=2$. The general case easily follows from induction. 
    \par Recall that, by definition, the second component of $(\Phi,\Psi)^{* 2}$ is given by $\Psi^{*2}$. Also, if $\rho'$ is the monotone cumulant of $\Psi$, i.e.~$\Psi = \exp^*(\rho')$, then monotone cumulant of the product $\Psi^{* 2}$ is $2\rho'$. With this and the definition of the c-monotone cumulant of $(\Phi,\Psi)^{* 2}$, we want to prove that the Boolean cumulant character of the first component of $(\Phi,\Psi)^{* 2}$ is given by $W_{2\rho'}(2\mathsf{P})$.
    \par We know that by the definition of c-monotone convolution, the Boolean cumulant of the first component of $(\Phi,\Psi)^{* 2}$ is given by $\beta + \Psi^{*-1}\succ \beta \prec \Psi$, where $\beta$ is the Boolean cumulant of $\Phi$, i.e.~$\Phi = \E_\succ(\beta)$. Moreover, the definition of $\mathsf{P}$ implies that $\beta = W_{\rho'}(\mathsf{P})$. Then, recalling that
$$
	\Psi^{*-1}\succ \beta \prec \Psi 
	= e^{r_{\lhd \rho'}} (\beta)
$$
from the proof of \cref{prop:orthogonalShuffle} and the identity $W_{\rho'}(\mathsf{P}) = \frac{e^{r_{\lhd \rho'}} -\id}{r_{\lhd \rho'}}(\mathsf{P})$ from \cref{not:exp}, we have that
\begin{eqnarray*}
     \beta + \Psi^{*-1}\succ \beta \prec \Psi
     &=& \beta + e^{r_{\lhd \rho'}}(\beta) \\ 
     &=& (e^{r_{\lhd \rho'}}+\id )(\beta)\\ 
     &=& (e^{r_{\lhd \rho'}}+\id )(W_{\rho'}(\mathsf{P}))\\ 
     &=& (e^{r_{\lhd \rho'}}+\id )\left( \frac{e^{r_{\lhd \rho'}}-\id}{r_{\lhd \rho'}}\right)(\mathsf{P})\\
     &=& 2 \frac{e^{2r_{\lhd \rho'}}-\id}{2r_{\lhd \rho'}}(\mathsf{P}).
\end{eqnarray*}
Notice that $2r_{\lhd \rho'}(\alpha) = 2\alpha \lhd \rho'= \alpha \lhd (2\rho') = r_{\lhd 2\rho'}(\alpha)$ for any $\alpha\in\mathfrak{g}$, and thus $2r_{\lhd \rho'} = r_{\lhd 2\rho'}$. Hence
$$
	\beta + \Psi^{*-1}\succ \beta \prec \Psi 
	=  2 \frac{e^{r_{\lhd 2\rho'}}-\id}{r_{\lhd 2\rho'}}(\mathsf{P}) 
	= 2W_{2\rho'}(\mathsf{P}) 
	= W_{2\rho'}(2\mathsf{P}),
$$
which concludes the proof for $N=2$.
\end{proof}


\section{A relation between conditional cumulants}
\label{sec:relation}

In this section, we describe a combinatorial formula that relates c-free cumulants and c-monotone cumulants. It is obtained by relating the corresponding infinitesimal characters in the shuffle-algebraic picture for non-commutative probability.

\begin{prop}
\label{prop:relation}
Let $(\Phi,\Psi)$ be a pair of characters on the double tensor Hopf algebra $T(T_+(\A))$ such that $\Phi = \E_\succ(\beta)$ and $\Psi = \exp^*(\rho')$. In addition, consider the c-monotone cumulant form $\mathsf{P}:=\mathsf{P}(\Phi,\Psi)$ and the c-free cumulant form $\mathsf{K}:=\mathsf{K}(\Phi,\Psi)$. Then we have that
$$
	\mathsf{K} = W_{-\rho'}(\mathsf{P}).
$$
\end{prop}

\begin{proof}
    Recall that $\mathsf{K} = \Psi\succ \beta \prec \Psi^{*-1}$ and $\beta = W_{\rho'}(\mathsf{P})$. From the proof of \cref{prop:orthogonalShuffle}, we know that 
 $$
    \Psi\succ \beta \prec \Psi^{*-1} = \exp^*(\rho')\succ\beta\prec \exp^*(-\rho') = e^{r_{\lhd -\rho'}}(\beta).
$$ 
Hence
    \begin{eqnarray*}
        \mathsf{K} 
        		&=&  \Psi\succ \beta \prec \Psi^{*-1}\\ 
		&=& e^{r_{\lhd -\rho'}}\left(W_{\rho'}(\mathsf{P}) \right)\\ 
		&=& e^{r_{\lhd -\rho'}}\left(\frac{e^{r_{\lhd \rho'}}-\id}{r_{\lhd \rho'}} \right)(\mathsf{P}).
    \end{eqnarray*}
    Noticing that $r_{\lhd -\rho'} = -r_{\lhd \rho'}$ implies that $e^{r_{\lhd -\rho'}} = e^{-r_{\lhd \rho'}}$, we get that
    \begin{eqnarray*}
        \mathsf{K}  
        &=& \left( \frac{e^{r_{\lhd -\rho'}}-\id}{-r_{\lhd \rho'}} \right) (\mathsf{P})\\ 
        &=&  \left( \frac{e^{r_{\lhd -\rho'}}-\id}{r_{\lhd -\rho'}} \right) (\mathsf{P})\\ 
        &=& W_{-\rho'}(\mathsf{P}),
    \end{eqnarray*}
    as we wanted to show.
\end{proof}

Applying \cref{lem:Wrho} together with the above proposition provides a new combinatorial formula relating c-free and c-monotone cumulants.

\begin{theo}
\label{cor:relation}
Let $(\A,\varphi,\psi)$ be a conditionally non-commutative probability space. If \linebreak $\{h_n^{(c)}: \A^n \to \mathbb{C}\}_{n \geq1}$ and $\{k_n^{(c)}: \A^n \to \mathbb{C}\}_{n \geq 1}$ are the c-monotone and c-free cumulants of $(\A,\varphi,\psi)$, respectively, and $\{h'_n: \A^n \to \mathbb{C}\}_{n \geq1}$ are the monotone cumulants of $(\A,\psi)$, then we have that
\begin{equation}
	k_n^{(c)}(a_1,\ldots,a_n) 
	= \sum_{\substack{\pi\in \NC^{\mathrm{irr}}(n)\\ 1\in \pi_1}} \frac{(-1)^{|\pi|-1}}{t(\pi)!} h_{|\pi_1|}^{(c)}
	(a_1,\ldots,a_n|\pi_1) \prod_{\substack{\pi_j \in \pi \\ j \neq 1}} h'_{|\pi_j|}(a_1,\ldots,a_n | \pi_j),
\end{equation}
for any $n\geq1$ and $a_1,\ldots,a_n\in \A$.
\end{theo}

Recalling that the operator $W_{-\rho'}$ is invertible with respect to composition, we can obtain the analogous formula of \eqref{eq:auxMB2} for the c-monotone case.

\begin{prop}
Let $(\A,\varphi,\psi)$ be a conditionally non-commutative probability space. If \linebreak $\{h_n^{(c)}: \A^n \to \mathbb{C}\}_{n \geq1}$ and $\{k_n^{(c)}: \A^n \to \mathbb{C}\}_{n \geq 1}$ are the c-monotone and c-free cumulants of $(\A,\varphi,\psi)$, respectively, and $\{k'_n: \A^n \to \mathbb{C}\}_{n \geq1}$ are the free cumulants of $(\A,\psi)$, then we have that
\begin{equation}
	h_n^{(c)}(a_1,\ldots,a_n) 
	= \sum_{\substack{\pi\in \NC^{\mathrm{irr}}(n)\\ 1\in \pi_1}} (-1)^{|\pi|-1}\omega(\pi) k_{|\pi_1|}^{(c)}
	(a_1,\ldots,a_n|\pi_1) \prod_{\substack{\pi_j \in \pi \\ j \neq 1}} k'_{|\pi_j|}(a_1,\ldots,a_n | \pi_j),
\end{equation}
for any $n\geq1$ and $a_1,\ldots,a_n\in \A$,  where $\omega(\pi)$ is the coefficient given in \eqref{eq:auxMB2}. 
\end{prop}
\begin{proof}
Let $(\Phi,\Psi)$ be the pair of characters on $T(T_+(\A))$ extending $(\varphi,\psi)$, and take $\mathsf{K}$ and $\mathsf{P}$ the corresponding c-free and c-monotone cumulant infinitesimal characters, respectively. Also, write $\Psi = \exp^*(\rho')$. From \cref{prop:relation}, we have that $\mathsf{K}=W_{-\rho'}(\mathsf{P})$. Using \cref{eq:MagnusExt}, the previous relation implies that $$\mathsf{P} = \Omega_{-\rho'}(\mathsf{K}) = \sum_{m\geq0}\frac{B_m}{m!}r^{(m)}_{\lhd-\rho'}(\mathsf{K}).$$
Let $w= a_1\cdots a_n \in \A^{\otimes n}$. Using \cref{lem:iterated}, we can compute the iterated pre-Lie product $r^{(m)}_{\lhd-\rho'}(\mathsf{K})$ and obtain:
\begin{eqnarray*}
h_n^{(c)}(a_1,\ldots,a_n) &=& \mathsf{P}(w)
\\ &=& \sum_{m\geq0}\frac{B_m}{m!}r^{(m)}_{\lhd-\rho'}(\mathsf{K})(w)
\\ &=& \sum_{m=0}^{n-1}\frac{B_m}{m!}(-1)^m \sum_{\substack{\pi\in\NCirr_{m+1}(n)\\1\in \pi_1}} \frac{|\pi|!}{t(\pi)!} \mathsf{K}(w_{\pi_1}) \prod_{\substack{\pi_j\in\pi\\j\neq 1}} \rho'(w_{\pi_j}),
\end{eqnarray*}
where we have written the sum over $\NCirr_{m+1}(n)$ instead of $\mathrm{M}^{\scriptstyle\mathrm{irr}}_{m+1}(n)$ by grouping the $\frac{|\pi|!}{t(\pi)!}$ monotone partitions associated with the same base non-crossing partition. Also, notice that the last sum above runs from $0$ to $n-1$ since $r^{(m)}_{\lhd-\rho'}(\mathsf{K})(w)=0$ for $m\geq n$. Now, by using \eqref{eq:auxMB2} on each $\rho'(w_{\pi_j}) = h'_{|\pi_j|}(a_1,\ldots,a_n|\pi_j)$, we have that
$$ 
\mathsf{P}(w)=   \sum_{\substack{\pi\in\NCirr(n)\\1\in \pi_1}} (-1)^{|\pi|-1}\frac{B_{|\pi|-1}}{t(\pi\backslash\{\pi_1\})!} \mathsf{K}(w_{\pi_1}) \prod_{\substack{\pi_j\in\pi\\j\neq 1}}\left(\sum_{\sigma_j \in \NCirr(\pi_j)} (-1)^{|\sigma_j|-1}\omega(\sigma_j)k'_{|\sigma_j|}(a_1,\ldots,a_n|\sigma_j)\right).
$$
Now we can conclude in the same way that in the proof of \cite[Thm.~3]{CEFPP}. We will outline the main ideas. First, it is possible to show that there is a bijection between
\begin{itemize}
\item pairs $\Big(\pi\backslash\{\pi_1\}, \{\sigma_{\pi_j}\}_{\pi_j\in \pi\backslash \{\pi_1\}}\Big)$, where $\pi\in\NCirr(n)$, $\pi_1$ is its
unique outer block and the $\sigma_{\pi_j}\in\NCirr(\pi_j)$ are irreducible non-crossing
partitions of the blocks of $\pi\backslash \{\pi_1\} = \{\pi_2,\pi_3,\ldots,\pi_\ell\}$; 
\item pairs $(\tau, S)$, where $\tau = \tau' \cup \tau_1\in \NCirr(n)$ is an irreducible non-crossing partition with outer block $\tau_1$ and $S$ is a subset of the set of blocks of $\tau'$ that includes all the outer blocks of $\tau'$.
\end{itemize}
Finally, using the above bijection and the recursion in \cite[Prop.~5]{CEFPP} for the $\omega(\pi)$ coefficients, we can write
$$
h_n^{(c)}(a_1,\ldots,a_n) = \sum_{\substack{\tau\in\NCirr(n)\\1\in \tau_1}} (-1)^{|\tau|-1} \omega(\tau) k^{(c)}(a_1,\ldots,a_n|\tau_1)\prod_{\substack{\tau_j\in\tau\\j\neq 1}} k'_{|\pi_j|}(a_1,\ldots,a_n|\tau_j),
$$
as we wanted to prove.
\end{proof}



\section{Relation with $t$-cumulants}
\label{sec:tboolean}

In this section, we describe how the relation between c-free cumulants and $t$-Boolean cumulants motivates the definition of an analog object related with c-monotone cumulants.


\subsection{$t$-Boolean cumulants via shuffle algebra} 

Let us recall the notion of $t$-Boolean cumulants, which was introduced  in \cite{BW01}  by Bo\.zejko and Wysoczanski as combinatorial objects that define an interpolation between free and Boolean cumulants. More precisely, given a non-commutative probability space $(\A,\varphi)$ and any $t\in [0,1]$, we define the sequence of multilinear functionals $\{b^{(t)}_n:\A^n\to\mathbb{C}\}_{n\geq1}$ recursively given by
\begin{equation}
	\varphi(a_1\cdots a_n) = \sum_{\pi\in\NC(n)} t^{\mathrm{inner}(\pi)} b^{(t)}_\pi(a_1,\ldots,a_n),
\end{equation}
for any $n\geq1$ and $a_1,\ldots,a_n\in \A$, where $\mathrm{inner}(\pi)$ stands for the number of inner blocks of $\pi\in\NC(n)$. The appearance of the coefficient $t^\mathrm{inner}(\pi)$ in the above formula prompts us to relate $t$-Boolean cumulants with the definition of c-free cumulants. 
\par Now, consider the double tensor Hopf algebra $T(T_+(\A))$ as well as the character $\Phi$ on it, extending $\varphi$. Also, let $\beta$ be the infinitesimal character associated to the Boolean cumulants of $\varphi$, i.e.~$\Phi = \E_\succ(\beta)$. Recall that for any other linear functional $\psi: \A \to \mathbb{C}$, the infinitesimal character $\mathsf{K}$ associated to the c-free cumulants of $(\varphi,\psi)$ is given by
$$
	\beta = \Psi^{*-1}\succ \mathsf{K} \prec\Psi,
$$
where $\Psi$ is the character on $T(T_+(\A))$ extending $\psi$. In particular, for any $t \in [0,1]$, we take the character $\Phi_t := \E_\succ(t\beta)$. Then, if we write $\beta^{(t)}$ as the infinitesimal character associated to the c-free cumulants of $(\Phi,\Phi_t)$, we have that 
$$
	t\beta^{(t)} 
	= \Phi_t \succ t\beta \prec \Phi_t^{*-1} 
	= \E_\succ(t\beta) \succ t\beta \prec \E_\succ(t\beta)^{*-1}.
$$
The previous equation implies that $\Phi_t = \E_\prec (t\beta^{(t)})$, i.e.~$t\beta^{(t)}$ is the infinitesimal character associated to the free cumulants of $\Phi_t$. Then, by \cref{lem:auxL15}, we have that
\begin{eqnarray*}
	\beta(w) 
	&=& \big(\Phi_t^{*-1}\succ \beta^{(t)}\prec \Phi_t\big)(w)\\ 
	&=& \sum_{\pi\in \NC^{\mathrm{irr}}(n)} t^{|\pi|-1}\beta^{(t)}_\pi(w).
\end{eqnarray*}
The previous development motivates the following definition.

\begin{defi}
For $\Phi$ a character on $T(T_+(\A))$ and $t \in [0,1]$, we define the \textit{$t$-Boolean cumulant infinitesimal character} to be $\beta^{(t)}$ on $T(T_+(\A))$ given by
\begin{equation}
\label{eq:tBoolmc}
\beta^{(t)} = \E_\succ(t\beta)\succ\beta\prec \E_\succ(t\beta)^{*-1}.
\end{equation}
\end{defi}

\noindent From the right shuffle fixed-point equation, we have that
\begin{eqnarray*}
	\Phi 
	&=& \epsilon + \Phi\succ\beta \\ 
	&=& \epsilon + \Phi\succ \big( \Phi_t^{*-1}\succ\beta^{(t)}\prec \Phi_t\big),
\end{eqnarray*}
where by evaluating on a word $w=a_1\cdots a_n \in \A^{\otimes n}$, we obtain
$$
	\Phi(w) = \sum_{\pi\in\NC(n)} t^{\mathrm{inner}(\pi)} \beta^{(t)}_\pi(w).
$$
This equation implies:

\begin{prop}
Let $(\A,\varphi)$ be a non-commutative probability space and for any $t\in[0,1]$, consider the infinitesimal character
 $\beta^{(t)}$ on $T(T_+(\A))$ defined in \cref{eq:tBoolmc}. Then for any word $w=a_1\cdots a_n\in T_+(\A)$, we have
$$
	\beta^{(t)}(w) = b^{(t)}_n(a_1,\ldots,a_n).
$$
In other words, $\beta^{(t)}$ is the infinitesimal character on $T(T_+(\A))$ associated to the $t$-Boolean cumulants $\{b^{(t)}\}_{n\geq1}$ of $(\A,\varphi)$.
\end{prop}

From the shuffle algebra point of view on $t$-Boolean cumulants seen as infinitesimal character, we can recover some properties of $t$-Boolean cumulants using shuffle-algebraic relations. For instance:

\begin{prop}[{\cite[Rmk.~7.9]{CENPW}}]
\label{prop:tsBoolean}
Let $(\A,\varphi)$ be a non-commutative probability space and for any $t\in [0,1]$, let $\{b_n^{(t)}: \A^n \to \mathbb{C} \}_{n\geq1}$ be the family of $t$-Boolean cumulants of $(\A,\varphi)$. Then for any $s,t\in[0,1]$, $n\geq1$, and $a_1,\ldots,a_n\in \A$, we have that
\begin{equation}
	b_n^{(t)}(a_1,\ldots,a_n) 
	= \sum_{\pi\in\NC^{\mathrm{irr}}(n)}(s-t)^{\mathrm{inner}(\pi)} b_\pi^{(s)}(a_1,\ldots,a_n).
\end{equation}
\end{prop}

\begin{proof}
Let $\Phi$ be the character on $T(T_+(\A))$ extending $\varphi$. For $s,t\in[0,1]$, let $\beta^{(s)}$ and $\beta^{(t)}$ be the infinitesimal characters associated to the $s$-Boolean and $t$-Boolean cumulants, respectively. In particular, we have the relations
\begin{eqnarray*}
	\beta 
	&=& \Phi_s^{*-1}\succ \beta^{(s)}\prec \Phi_s,\\
	\beta^{(t)} 
	&=& \Phi_t\succ\beta\prec \Phi_t^{*-1},
\end{eqnarray*}
where $\Phi_t = \E_\succ(t\beta)$ and $\Phi_s= \E_\succ(s\beta)$. Combining both equations we obtain
\begin{eqnarray*}
	\beta^{(t)}
	&=& \Phi_t \succ\beta\prec\Phi_t^{*-1}\\ 
	&=& \Phi_t \succ\big( \Phi_s^{*-1}\succ \beta^{(s)} \prec\Phi_s\big)\prec \Phi_t^{*-1}\\ 
	&=& \big(\Phi_s*\Phi_t^{*-1}\big)^{*-1}\succ \beta^{(s)}\prec \Phi_s*\Phi_t^{*-1},
\end{eqnarray*}
where we used the shuffle identities in the third equality. It is not difficult to show from the definition that if $\alpha\in\mathfrak{g}$, then $\E_\prec(\alpha)^{*-1} = \E_\succ(-\alpha)$ (see \cite[Lem.~2]{EFP1}). Together with the relation $\Phi_s =\E_\prec(s\beta^{(s)})$, we compute
\begin{eqnarray*}
	\Phi_s*\Phi_t^{*-1}
	&=& \E_\prec(s\beta^{(s)})* \E_\succ(t\beta)^{*-1}\\ 
	&=&  \E_\prec(s\beta^{(s)})* \E_\prec(-t\beta).
\end{eqnarray*}
On the other hand, if $\Psi_1 = \E_\prec(\alpha_1)$ and $\Psi_2= \E_\prec(\alpha_2)$ are two characters on $T(T_+(\A))$, then it can be shown from the shuffle identities that
\begin{equation}
\label{eq:subordination}
	\Psi_1\boxplus \Psi_2 
	= \E_\prec(\alpha_1+\alpha_2) = \E_\prec(\alpha_1)*\E_\prec( \Psi_1^{*-1} \succ \alpha_2 \prec \Psi_1),
\end{equation}
see, for instance \cite[Thm.~31]{EFP3}. With this identity in hand and noticing that $$-t\beta = \Phi_s^{*-1}\succ (-t\beta^{(s)})\prec \Phi_s,$$ we obtain
$$
	\Phi_s*\Phi_t^{*-1} 
	=  \E_\prec(s\beta^{(s)})* \E_\prec(  \Phi_s^{*-1} \succ (-t\beta^{(s)}) \prec \Phi_s ) 
	= \E_\prec(s\beta^{(s)}- t\beta^{(s)})  = \E_\prec((s-t)\beta^{(s)}),
$$
i.e.~$(s-t)\beta^{(s)}$ is the infinitesimal character associated to the free cumulants of $\Phi_s*\Phi_t^{*-1}$.  Finally applying \cref{lem:auxL15}, we have for a word $w=a_1\cdots a_n\in \A^{\otimes n}$:
\begin{eqnarray*}
	\beta^{(t)}(w) 
	&=& \left( \big(\Phi_s*\Phi_t^{*-1}\big)^{*-1}\succ \beta^{(s)}\prec \Phi_s*\Phi_t^{*-1}\right)(w)\\ 
	&=& \sum_{\pi\in \NC_{\mathrm{irr}}(n)} (s-t)^{|\pi|-1}\beta^{(s)}_\pi(w),
\end{eqnarray*}
and we conclude by identifying the evaluation $\beta^{(t)}(w)$ with the $t$-Boolean cumulant $b_n^{(t)}(a_1,\ldots,a_n)$.
\end{proof}

\begin{rem}
In the proof of the previous proposition,  in particular in \eqref{eq:subordination}, we have written the free additive convolution of two characters as the monotone additive convolution of two characters.  In the language of non-commutative probability, we can recall the \textit{subordination convolution} introduced by Lenczewzki in \cite{Len}. More precisely, for any two distributions $\mu$ and $\nu$, the subordination convolution of $\nu$ and $\mu$ is the distribution $\nu\boxright \mu$  such that 
$$
	\mu\boxplus \nu = \mu \blacktriangleright (\nu \boxright\mu).
$$
On the other hand, for any two characters $\Psi_1$ and $\Psi_2$ in $T(T_+(\A))$, we can define the character $\Psi_2\boxright \Psi_1$ by the recipe $\Psi_2\boxright\Psi_1 := \E_\prec\big(\Psi_1^{*-1}\succ \alpha_2 \prec \Psi_1 \big)$, where $\Psi_2 = \E_\prec(\alpha_2)$. By \eqref{eq:subordination}, we have that $ \Psi_1\boxplus\Psi_2 = \Psi_1 * (\Psi_2\boxright\Psi_1)$. If in particular $\Psi_1$ and $\Psi_2$ are the characters extending the distributions $\mu$ and $\nu$ of random variables on a non-commutative probability space $(\A,\varphi)$, then we can show that $\Psi_2\boxright\Psi_1$ is the character on $T(T_+(\A))$ extending the subordination convolution $\nu\boxright\mu$.
\end{rem}

Now, we explain the relation of $t$-Boolean cumulants with the Belinschi--Nica semigroup of maps defined in \cite{BN06}. Given $t\geq0$ and a character $\Phi$ on the double tensor Hopf algebra $T(T_+(\A))$, we have a family of maps $\{\mathbb{B}_t\}_{t\geq0}$, where $\Phi\mapsto \mathbb{B}_t(\Phi)$ is a new character given by (\cite[Lem.~42]{EFP3})
\begin{equation}
	\mathbb{B}_t(\Phi) 
	= \E_\prec\big( \E_\prec(t\kappa)^{*-1}\succ\kappa\prec \E_\prec(t\kappa) \big),
\end{equation}
where $\Phi = \E_\prec(\kappa)$. The next result tells how to re-interpret the above equation using the right half-shuffle logarithm $\beta = \mathcal{L}_\succ(\Phi)$ instead (see \cref{rem:shuffleLogs}).

\begin{lemma}
\label{lem:tBPBij}
Consider the double tensor Hopf algebra $T(T_+(\A))$. If $\Phi$ is a character on $T(T_+(\A))$ such that $\Phi = \E_\succ(\beta)$, then for any $t\geq0$, we have
\begin{equation}
	\mathbb{B}_t(\Phi) = \E_\succ\left( \E_\prec(t\beta)\succ\beta\prec  \E_\prec(t\beta)^{*-1}\right).
\end{equation}
\end{lemma}

\begin{proof}
Let $\Phi$ be a character on $T(T_+(\A))$ and take $t\geq0$. Let $\beta$ and $\kappa$ be the infinitesimal characters such that $\Phi = \E_\prec(\kappa) = \E_\succ(\beta)$. Recalling that $\kappa = \E_{\prec}(\kappa)\succ \beta\prec \E_\prec(\kappa)^{*-1}$ and writing $\theta_{\E_\prec(t\kappa)}(\kappa):=   \E_\prec(t\kappa)^{*-1}\succ\kappa\prec \E_\prec(t\kappa) $, we have that
\begin{eqnarray*}
	\mathbb{B}_t(\Phi) 
	&=& \E_\prec\big( \theta_{\E_\prec(t\kappa)}(\kappa)  \big)\\ 
	&=& \E_\succ\big( \mathcal{L}_\succ \big( \E_\prec( \theta_{\E_\prec(t\kappa)}(\kappa) )\big)\big)\\
	&=& \E_\succ\big(  \E_\prec( \theta_{\E_\prec(t\kappa)}(\kappa))^{*-1}
	\succ \theta_{\E_\prec(t\kappa)}(\kappa) \prec \E_\prec( \theta_{\E_\prec(t\kappa)}(\kappa)) \big)\\ 
	&=& \E_\succ\big(  \E_\prec( \theta_{\E_\prec(t\kappa)}(\kappa))^{*-1}*\E_\prec(t\kappa)^{*-1} 
	\succ\kappa \prec \E_\prec(t\kappa)* \E_\prec( \theta_{\E_\prec(t\kappa)}(\kappa)) \big) ,
\end{eqnarray*}
where we used the shuffle identities in the last equality. Also, using \eqref{eq:subordination} twice, we obtain
\begin{eqnarray*}
	\E_\prec(t\kappa)* \E_\prec( \theta_{\E_\prec(t\kappa)}(\kappa)) &=& \E_\prec((t+1)\kappa) \\ 
	&=& \E_\prec(\kappa)*\E_\prec\big( \E_\prec(\kappa)\succ t\kappa \prec \E_\prec(\kappa)\big) \\ 
	&=& \E_\prec(\kappa)* \E_\prec(t\beta) .
\end{eqnarray*}
Finally, we compute
\begin{eqnarray*}
	\mathbb{B}_t(\Phi) 
	&=& \E_\succ\big(\E_\prec(t\beta)^{*-1}* \E_\prec(\kappa)^{*-1}\succ \kappa\prec  \E_\prec(\kappa)* \E_\prec(t\beta) \big)\\ 
	&=&  \E_\succ\big(\E_\prec(t\beta)^{*-1} \succ \big(\E_\prec(\kappa)^{*-1} 
	\succ \kappa\prec \E_\prec(\kappa) \big) \prec \E_\prec(t\beta) \big)\\ 
	&=& \E_\succ\big( \E_\prec(t\beta)^{*-1}\succ\beta\prec \E_\prec(t\beta) \big),
\end{eqnarray*}
and the proof is complete.
\end{proof}

A natural question follows by considering the relation between the distribution $\mathbb{B}_t(\Phi)$ and $t$-Boolean cumulants. In this situation, we can prove a claim remarked in \cite[Rmk.~11.9]{CENPW} which states that $t$-Boolean cumulants of $\mathbb{B}_t(\Phi)$ are given precisely by the Boolean cumulants of $\Phi$.

\begin{prop}[{\cite[Rmk.~11.9]{CENPW}}]
Let $(\A,\varphi)$ be a non-commutative probability space. Consider the double tensor Hopf algebra $T(T_+(\A))$ and the character $\Phi$ on it, extending $\varphi$, and write $\Phi = \E_\succ(\beta)$. Then for any $t\in[0,1]$, the $t$-Boolean infinitesimal character of $\mathbb{B}_t(\Phi)$ is given by $\beta$.
\end{prop}

\begin{proof}
Fix $t\in[0,1]$. By \cref{lem:tBPBij}, we have
$$
	\mathbb{B}_t(\Phi) = \E_\succ\left( \E_\prec(t\beta)^{*-1}\succ\beta\prec  \E_\prec(t\beta)\right).
$$
Define $\gamma_t :=  \E_\prec(t\beta)^{*-1}\succ\beta\prec  \E_\prec(t\beta)$. If $\lambda^{(t)}$ is the $t$-Boolean infinitesimal character of $\mathbb{B}_t(\Phi)$, by definition we have that
$$
	\gamma_t = \Psi_t^{*-1}\succ \lambda^{(t)}\prec \Psi_t,
$$
where $\Psi_t:= \E_\succ (t\gamma_t) = \E_\prec(t\lambda^{(t)})$. From the first equality, we compute
\begin{eqnarray*}
	\Psi_t
	&=& \E_\succ(t\gamma_t)\\ 
	&=& \E_\succ\big(  \E_\prec(t\beta)^{*-1}\succ t\beta\prec  \E_\prec(t\beta)\big)\\ 
	&=& \E_\prec(t\beta),
\end{eqnarray*}
where we used the shuffle equation $\E_\prec(\alpha) = \E_\succ(\E_\prec(\alpha)^{*-1}\succ \alpha \prec\E_\prec(\alpha))$, for any $\alpha$ infinitesimal character. By comparing $\E_\prec(t\lambda^{(t)}) = \Psi_t = \E_\prec(t\beta)$, we conclude that $\lambda^{(t)} = \beta$, i.e.~the $t$-Boolean cumulants of $\mathbb{B}_t(\Phi)$ are given by the Boolean cumulants of $\Phi$, as we wanted to prove.
\end{proof}


\subsection{$t$-monotone cumulants} 

As final part of this section, we propose the analogous notion of $t$-Boolean cumulants in the case of c-monotone convolution. Recall that $t$-Boolean cumulants can be seen as a particular case of c-free cumulants. More precisely, if $(\A,\varphi)$ is a non-commutative probability space, $T(T_+(\A))$ is the double tensor Hopf algebra and $\Phi$ is the character on $T(T_+(\A))$ extending $\varphi$, then the infinitesimal character extending the $t$-Boolean cumulants is the c-free infinitesimal character of $(\Phi,\Phi_t)$. This prompts us to the following definition.

\begin{defi}
For a character on $T(T_+(\A))$ and $t\in [0,1]$, we define the \textit{$t$-monotone cumulant infinitesimal character} to be $\rho^{(t)}$ on $T(T_+(\A))$ given by the c-monotone infinitesimal character of $(\Phi,\Phi_t)$, where $\Phi_t = \E_\succ(t\beta)$ with $\Phi = \E_\succ(\beta)$.
\end{defi}

By definition of c-monotone infinitesimal character, we have $\rho^{(t)} = \mathsf{P}(\Phi,\Phi_t)$ such that \linebreak$\beta = W_{\rho'(t)}(\rho^{(t)})$, where $\exp^{*}(\rho'(t)) = \Phi_t$. On the other hand, by the pre-Lie Magnus expansion, we have that $t\beta = W(\rho'(t)) = W_{\rho'(t)}(\rho'(t))$. Since $W_{\rho'(t)}$ is a linear operator, we have that 
$$
	W_{\rho'(t)}(\rho^{(t)}) = \beta = W_{\rho'(t)}(\rho'(t)/t),
$$
which implies that $\rho'(t) = t\rho^{(t)}$ since the operator $W_{\rho'(t)}$ is invertible with respect to composition. Hence, if we define the \textit{$t$-monotone cumulants} as the elements of the family of multilinear forms $\{h^{(t)}_n:\A^n\to\mathbb{C}\}_{n\geq1}$ by the recipe $h_n^{(t)}(a_1,\ldots,a_n) = \rho^{(t)}(w)$, for any random variables $a_1,\ldots,a_n\in \A$ and  the word $w= a_1\cdots a_n\in \A^{\otimes n}$, we can use  \eqref{eq:lemma} and \cref{prop:46P} and arrive to the following identities. 

\begin{theo}
\label{thm:tmonotone}
Let $(\A,\varphi)$ be a non-commutative probability space and let $\{b_n:\A^n\to\mathbb{C}\}_{n\geq1}$ and $\{h_n^{(t)}:\A^n\to\mathbb{C}\}_{n\geq1}$ be Boolean and $t$-monotone cumulants of $(\A,\varphi)$, respectively. Then, for any $n\geq1$ and $a_1,\ldots,a_n\in \A$, we have that
\begin{eqnarray*}
b_n(a_1,\ldots,a_n) &=& \sum_{{\NCirr(n)}} \frac{t^{|\pi|-1}}{t(\pi)!} h^{(t)}_\pi(a_1,\ldots,a_n),\\
\varphi(a_1\cdots a_n) &=&  \sum_{{\mathrm{NC}(n)}} \frac{t^{\mathrm{inner}(\pi)}}{t(\pi)!} h^{(t)}_\pi(a_1,\ldots,a_n).
\end{eqnarray*}
\end{theo}

\allowdisplaybreaks
\begin{rem}
A natural question appears when considering the $t$-monotone analogue of \cref{prop:tsBoolean}. In other words, for any $t,s\geq0$, we would like to describe the coefficients $c^\pi(t,s)$ such that
\begin{equation}
\label{eq:tsmonotone1}
h_n^{(t)}(a_1,\ldots,a_n) = \sum_{\pi\in\NCirr(n)} c^\pi(t,s) h^{(s)}_\pi(a_1,\ldots,a_n),
\end{equation}
for any $n\geq1$ and $a_1,\ldots,a_n\in \A$. Observe that the case $(t,s) = (0,1)$ corresponds to \cref{eq:auxBM}, where the coefficient is given by $c^\pi(t,s) = 1/t(\pi)!$. More generally, the case $(t,s) = (0,s)$ corresponds to the first formula in \cref{thm:tmonotone}, where $c^\pi(t,s) = s^{|\pi|-1}/t(\pi)!$. On the other hand, the case $(t,s) = (1,0)$ corresponds to \eqref{eq:auxMB1}, where $c^\pi(t,s) = \omega(\pi)$. In the language of infinitesimal characters, \eqref{eq:tsmonotone1} can be read as
$$
\rho^{(t)} = \Omega_{t\rho^{(t)}}(\beta) = \Omega_{t\rho^{(t)}}\big(W_{s\rho^{(s)}}(\rho^{(s)})\big) = \frac{r_{\lhd t\rho^{(t)}}}{e^{r_{\lhd t\rho^{(t)}}}-\id}\frac{e^{r_{\lhd s\rho^{(s)}}}-\id}{r_{\lhd s\rho^{(s)}}}\big(\rho^{(s)}\big). 
$$
We can compute the terms of the first four orders of the above expansion:
\begin{eqnarray*}
\rho^{(t)} &=& \rho^{(s)} - \frac{1}{2}(t-s) \big(\rho^{(s)}\lhd \rho^{(s)}\big) + \left( \frac{1}{12}t^2 - \frac{1}{4}ts \right) \big( (\rho^{(s)}\lhd\rho^{(s)})\lhd \rho^{(s)}\big)
\\& & + \left( \frac{1}{4}t^2 - \frac{1}{4}ts + \frac{1}{6} s^2\right) \big(\rho^{(s)}\lhd ( \rho^{(s)}\lhd \rho^{(s)} )\big) - \left(\frac{1}{24}t^3 - \frac{1}{8}t^2s\right) \big( \rho^{(s)}\lhd ((\rho^{(s)}\lhd \rho^{(s)})\lhd \rho^{(s)})\big)
\\ & & -\left(\frac{1}{8}t^3 - \frac{1}{8}t^2s  + \frac{1}{12} ts^2 - \frac{1}{24}s^3 \right)  \big(\rho^{(s)}\lhd (\rho^{(s)}\lhd (\rho^{(s)}\lhd \rho^{(s)}))\big)
\\ & & -\left(\frac{1}{24}t^3 - \frac{1}{6}t^2 s + \frac{1}{8}ts^2 \right) \big((\rho^{(s)}\lhd \rho^{(s)}) \lhd (\rho^{(s)}\lhd \rho^{(s)})\big)
\\ & & - \left(\frac{1}{24}t^3 - \frac{1}{24}t^2s + \frac{1}{12} ts^2\right)\big((\rho^{(s)}\lhd (\rho^{(s)}\lhd \rho^{(s)}) )\lhd \rho^{(s)}\big)+\cdots
\end{eqnarray*}
By evaluating the previous relation on a word $w= a_1\cdots a_n\in \A^{\otimes n}$ and using \eqref{eq:iterated}, we can obtain a combinatorial formula, in terms of irreducible non-crossing partitions, that relates $t$-monotone and $s$-monotone cumulants, for $t,s\geq0$.
\end{rem}

\end{document}